\documentclass[12pt]{amsart}
\usepackage{amssymb,latexsym}
\usepackage[all]{xy}
\usepackage{graphicx}

\numberwithin{equation}{section}

\newtheorem{theorem}{Theorem}[section]
\newtheorem{proposition}[theorem]{Proposition}

\newtheorem{corollary}[theorem]{Corollary}
\newtheorem{lemma}[theorem]{Lemma}

\theoremstyle{definition}

\newtheorem{remark}[theorem]{Remark}
\newtheorem{example}[theorem]{Example}
\newtheorem{definition}[theorem]{Definition}

\newcommand{\prinA}{\mathcal{A}_{\bullet}}
\newcommand{\Yhats}{\hat{Y}_{1}, \ldots, \hat{Y_{n}}}
\newcommand{\basering}{\mathbb{Z}[q^{\pm \frac{1}{2}}]}

\def\aa{\mathbf{a}}
\def\bb{\mathbf{b}}
\def\cc{\mathbf{c}}
\def\dd{\mathbf{d}}
\def\ee{\mathbf{e}}
\def\ff{\mathbf{f}}
\def\gg{\mathbf{g}}
\def\hh{\mathbf{h}}
\def\xx{\mathbf{x}}

\def\ZZ{\mathbb{Z}}
\def\CC{\mathbb{C}}
\def\PP{\mathbb{P}}

\def\QQ{\mathbb{Q}}
\def\TT{\mathbb{T}}

\def\Fcal{\mathcal{F}}

\def\Xcal{\mathcal{X}}

\def\Qsf{\QQ_{\,\rm sf}}

\def\Trop{\operatorname{Trop}}

\def\sgn{\operatorname{sgn}}
\def\rank{\operatorname{rank}}

\newcommand{\overunder}[2]{
\!\begin{array}{c}
\scriptstyle{#1}\\[-.1in]
-\!\!\!-\!\!\!-\\[-.1in]
\scriptstyle{#2}
\end{array}
\!
}

\renewcommand{\eqref}[1]{{\rm (\ref{#1})}}

\begin{document}


\title[F-polynomials in Quantum Cluster Algebras]
{F-polynomials in Quantum Cluster Algebras
}

\author{Thao Tran}
\address{Department of Mathematics, Northeastern University,
Boston, MA 02115, USA} \email{tran.thao1@neu.edu}


\begin{abstract}
$F$-polynomials and $\gg$-vectors were defined by Fomin and Zelevinsky  to give a formula which expresses cluster variables in a cluster algebra in terms of the initial cluster data.  A quantum cluster algebra is a certain noncommutative deformation of a cluster algebra.   In this paper, we define and prove the existence of analogous quantum $F$-polynomials for quantum cluster algebras.  We prove some properties of quantum $F$-polynomials. In particular, we give a recurrence relation which can be used to compute them.  Finally, we compute quantum $F$-polynomials and $\gg$-vectors for a certain class of cluster variables, which includes all cluster variables in type $\mbox{A}_{n}$ quantum cluster algebras.
\end{abstract}

\date{April 21, 2009}

\thanks{Research supported by the NSF grants DMS-0500534 and DMS-0801187.}

\maketitle

\tableofcontents

\section{Introduction}

Cluster algebras were introduced by Fomin and Zelevinsky in \cite{ca1} in order to study total positivity and canonical bases in semisimple groups.  Let $m \geq n$ be positive integers.  Roughly speaking, a \emph{cluster algebra} is a subalgebra generated by a distinguished collection of  generators called \emph{cluster variables} inside of an ambient  field $\mathcal{F}$ which is isomorphic to the field of rational functions in $m$ independent variables.   To obtain these cluster variables, one begins with an initial \emph{seed}.  A \emph{seed} is a pair $(\tilde{\xx}, \tilde{B})$ such that $\tilde{\xx}$ is an $m$-tuple of elements from $\mathcal{F}$ with the first $n$ terms being cluster variables and the remaining $m - n$ terms being \emph{coefficient variables}, and such that $\tilde{B}$ is an $m \times n$ integer matrix whose  top $n \times n$ submatrix is skew-symmetrizable; the $m$-tuple $\tilde{\xx}$ is called the \emph{extended cluster} of the seed.   \emph{Seed mutations} are certain operations which transform a seed into another seed.   In mutating from one seed $(\tilde{\xx}, \tilde{B})$ to another seed $(\tilde{\xx}', \tilde{B}')$, one cluster variable $x$ is exchanged for another cluster variable $x'$; the elements $x, x'$ satisfy a certain \emph{exchange relation} which is of the form $xx' = M^{+} + M^{-}$, where $M^{+}, M^{-}$ are monomials on disjoint subsets of  variables from from $\tilde{\xx} - \{ x \}$.  The role of the matrix $\tilde{B}$ is to dictate exactly what these monomials $M^{+}, M^{-}$ are.   The set of cluster variables which generate a given cluster algebra is obtained by mutating the initial seed with all possible sequences of mutations and taking all cluster variables from the seeds that result.  

Quantum cluster algebras were defined by Berenstein and Zelevinsky in \cite{quantum}.  A \emph{quantum cluster algebra} is a certain noncommutative deformation of a cluster algebra with an additional generator $q$ lying in the center.  Under this deformation, each extended cluster $(x_{1}, \ldots, x_{m})$ in the cluster algebra is replaced by the $m$-tuple $(X_{1}, \ldots, X_{m})$, where the elements $X_{1}, \ldots, X_{m}$ now \emph{quasi-commute}, i.e., for each $1 \leq i, j \leq m$, there exists $\lambda_{ij} \in \ZZ$ such that $X_{i}X_{j} = q^{\lambda_{ij}}X_{j}X_{i}$.     The exchange relations are also altered to allow for the fact that extended cluster elements must quasi-commute.

In \cite{coefficients}, Fomin and Zelevinsky defined $F$-polynomials and $\gg$-vectors corresponding to an initial $n \times n$ skew-symmetrizable integer matrix $B^{0}$.  They gave recurrence relations for the $F$-polynomials and $\gg$-vectors which essentially only depend on $B^{0}$.  One of the main results of \cite{coefficients} was a formula expressing any cluster variable in terms of $F$-polynomials and $\gg$-vectors using only the initial cluster of $\mathcal{A}$.

The main purpose of this paper is to demonstrate the existence of \emph{quantum $F$-polynomials}, which is an analogue of $F$-polynomials in the quantum cluster algebra setting.  Quantum $F$-polynomials satisfy the property that any cluster variable in a quantum cluster algebra may be computed (up to a multiple of a power of $q$) in a formula with the appropriate quantum $F$-polynomial and $\gg$-vector using only the initial extended cluster.   It is conjectured (and proven in some cases) that this formula can be sharpened so that the multiple of $q$ does not appear (see Theorem \ref{thm:no-q-factor}).   By setting $q = 1$ in a quantum $F$-polynomial, we obtain the appropriate $F$-polynomial for (nonquantum) cluster algebras.

The organization of the paper is as follows.   In Section \ref{section:ca}, we recall the definition of cluster algebras.   In Section \ref{section:f-polys}, we recall the definition of $F$-polynomials and $\gg$-vectors as well as the aforementioned formula for cluster variables (Theorem \ref{thm:cluster-var-formula}).   Section \ref{section:quantum} is devoted to recalling the definition of quantum cluster algebras.    In Section \ref{section:quantum-f-polys}, (left) quantum $F$-polynomials are defined.  Theorem \ref{thm:quantum-fpoly} is devoted to proving their existence.    Properties of quantum $F$-polynomials are given in Section \ref{section:quantum-f-polys-props}.    Proposition \ref{prop:fpoly-right} shows how to easily compute "right" quantum $F$-polynomials once the "left" ones are known.  A recurrence relation for quantum $F$-polynomials is given in Theorem \ref{thm:quantum-fpoly-rec}.  The last section is devoted to computing quantum $F$-polynomials in a particular class of examples where the cluster variables correspond to induced trees of the quiver which can be defined using the initial matrix $B^{0}$.    We show in the subsection that this formula may be used to compute all quantum $F$-polynomials corresponding to any $n \times n$ exchange matrix $B^{0}$ of type $\mbox{A}_{n}$.  In addition, a formula for $\gg$-vectors in type $\mbox{A}_{n}$ is given.  In the sequel to this paper, formulas for $\gg$-vectors and quantum $F$-polynomials will be provided for the remainder of the classical types.

\section{Cluster Algebras of Geometric Type} \label{section:ca}
Following \cite[Section 2]{coefficients}, we give the definition of a cluster algebra of geometric type as well as some properties of these cluster algebras.  The proofs of any statements given in this section can be found in \cite[Section 2]{coefficients}.
\begin{definition} \label{def:semifield-tropical}
Let $J$ be a finite set of labels,
and let $\Trop (u_j: j \in J)$ be an abelian group (written
multiplicatively) freely generated by the elements $u_j \, (j \in J)$.
We define the addition~$\oplus$ in $\Trop (u_j: j \in J)$ by
\begin{equation}
\label{eq:tropical-addition}
\prod_j u_j^{a_j} \oplus \prod_j u_j^{b_j} =
\prod_j u_j^{\min (a_j, b_j)} \,,
\end{equation}
and call $(\Trop (u_j: j \in J),\oplus,\cdot)$ a \emph{tropical
 semifield}.   If $J$ is empty,  we obtain the trivial semifield consisting of a single element~$1$.   The group ring of $\Trop (u_j: j \in J)$ is the ring of Laurent polynomials in the variables~$u_j\,$.
\end{definition}

Fix two positive integers $m$, $n$ with $m \geq n$.  Let $\mathbb{P} = \Trop (x_{n + 1}, \ldots, x_{m})$, and let $\mathcal{F}$ be the field of rational functions in $n$ independent variables with coefficients in $\QQ\PP$, the field of fractions of the integral group ring $\ZZ\PP$.  (Note that the definition of $\mathcal{F}$ does not depend on the auxiliary addition $\oplus$ in $\mathbb{P}$.)  The group ring $\ZZ\mathbb{P}$ will be the ground ring for the cluster algebra $\mathcal{A}$ to be defined, and $\mathcal{F}$ will be the ambient field, with $n$ being the \emph{rank} of $\mathcal{A}$.

\begin{definition} \label{def:labeled-seed}
A \emph{labeled seed} in $\mathcal{F}$ is a pair $(\tilde{\xx}, \tilde{B})$ where
\begin{itemize}
  \item $\tilde{\xx} = (x_{1}, \ldots, x_{m})$, where $x_{1}, \ldots, x_{n}$ are algebraically independent over $\QQ\PP$, and $\mathcal{F} = \QQ\PP(x_{1}, \ldots, x_{n})$, and 
  \item $\tilde{B}$ is an $m \times n$ integer matrix such that the submatrix $B$ consisting of the top $n$ rows and columns of $\tilde{B}$ is skew-symmetrizable (i.e., $DB$ is skew-symmetric for some $n \times n$ diagonal matrix $D$ with positive integer diagonal entries).
\end{itemize}
We call $\tilde{\xx}$ the \emph{extended cluster} of the labeled seed, $(x_{1}, \ldots, x_{n})$ the \emph{cluster}, $\tilde{B}$ the \emph{exchange matrix}, and the matrix $B$ the \emph{principal part} of $\tilde{B}$.
\end{definition}
We fix some notation to be used throughout the paper.  For $x \in \QQ$,  
\begin{align*}
[x]_+ &= \max(x,0); \\
\sgn(x) &=
\begin{cases}
-1 & \text{if $x<0$;}\\
0  & \text{if $x=0$;}\\
 1 & \text{if $x>0$;}
\end{cases}\\
\end{align*}
Also, for $i, j \in \ZZ$, write $[i, j]$ for the set $\{ k \in \ZZ : i \leq k \leq j \}$.  In particular, $[i, j] = \varnothing$ if $i > j$.

\begin{definition} \label{def:matrix-mut} Let $k \in [1, n]$.  We say that an $m \times n$ matrix $\tilde{B}'$ is obtained from an $m \times n$ matrix $\tilde{B} = (b_{ij})$ by \emph{matrix mutation} in direction $k$ if the entries of $\tilde{B}'$ are given by 
\begin{equation}
\label{eq:matrix-mutation}
b'_{ij} =
\begin{cases}
-b_{ij} & \text{if $i=k$ or $j=k$;} \\[.05in]
b_{ij} + \sgn(b_{ik}) \ [b_{ik}b_{kj}]_+
 & \text{otherwise.}
\end{cases}
\end{equation}
\end{definition} 

\begin{definition} \label{def:seed-mut} Let $(\tilde{\xx}, \tilde{B})$ be a labeled seed in $\mathcal{F}$ as in Definition \ref{def:labeled-seed}, and write $\tilde{B} = (b_{ij})$.  The \emph{seed mutation} $\mu_{k}$ in direction $k$ transforms $(\tilde{\xx}, \tilde{B})$ into the labeled seed $\mu_{k}(\tilde{\xx}, \tilde{B}) = (\tilde{\xx}', \tilde{B}')$, where 
\begin{itemize}
   \item $\tilde{\xx}'=(x_1',\dots,x_m')$, where $x'_{j} = x_{j}$ for $j \neq k$, and 
      \begin{equation}
\label{eq:exchange-relation}
x'_k = x_k^{-1}\left(
\displaystyle\prod_{i=1}^m x_i^{[b_{ik}]_+}
+ \displaystyle\prod_{i=1}^m x_i^{[-b_{ik}]_+}\right) ,
\end{equation}
   \item $\tilde{B}'$ is obtained from $\tilde{B}$ by matrix mutation in direction $k$.
\end{itemize}
\end{definition}
One may check that the pair $(\tilde{\xx}', \tilde{B}')$ obtained is again a labeled seed.  Furthermore, the seed mutation $\mu_{k}$ is involutive, i.e., applying $\mu_{k}$ to $(\tilde{\xx}, \tilde{B})$ yields the original labeled seed $(\tilde{\xx}, \tilde{B})$.

\begin{definition} Let $\mathbb{T}_{n}$ be the $n$-regular tree whose edges are labeled with $1, \ldots, n$ in such a way that for each vertex, the $n$ edges emanating from that vertex each receive different labels.  Write $t \frac{k}{\hspace{1cm}} t'$ to indicate $t, t' \in \mathbb{T}_{n}$ are joined by an edge with label $k$.
\end{definition}

\begin{definition} \label{def:cluster-pattern} A \emph{cluster pattern} is an assignment of a labeled seed $(\tilde{\xx}_{t}, \tilde{B}_{t})$ to every vertex $t \in \mathbb{T}_{n}$ such that if $t \frac{k}{\hspace{1cm}} t'$, then the labeled seeds assigned to $t$, $t'$ may be obtained from one another by seed mutation in direction $k$.   Write $\tilde{\xx}_{t}  = (x_{1; t}, \ldots, x_{m; t})$, $\tilde{B}_{t}  =  (b^t _{ij})$, and denote by $B_{t}$ the principal part of $\tilde{B}_{t}$.  
\end{definition}

\begin{definition} For a given cluster pattern, write 
\begin{equation}
\Xcal = \{ x_{j; t} : j \in [1, n], t \in \mathbb{T}_{n} \}.
\end{equation}
The elements of $\Xcal$ are the \emph{cluster variables}.  The \emph{cluster algebra} $\mathcal{A}$ associated to this cluster pattern is the $\ZZ\PP$-subalgebra of $\mathcal{F}$ generated by all cluster variables.  That is, $\mathcal{A} = \ZZ\PP[\Xcal]$.
\end{definition}

\section{$F$-polynomials and $\gg$-vectors} \label{section:f-polys}

The reference for all of the information given in this section (except the definition and properties of extended $\gg$-vectors)  is \cite{coefficients}.   For this section, fix an $n \times n$ skew-symmetrizable integer matrix $B^{0}$ and an initial vertex $t_{0} \in \mathbb{T}_{n}$.  Assume that any cluster algebra $\mathcal{A}$ in this section has initial exchange matrix $\tilde{B}^{0}$ with principal part $B^{0}$.

\begin{definition} We say that a cluster pattern $t \mapsto (\tilde{\xx}_{t}, \tilde{B}_{t})$ (or its corresponding cluster algebra) has \emph{principal coefficients} at $t_{0}$ if $\tilde{\xx}_{t_{0}} = (x_{1}, \ldots, x_{2n})$ (i.e., $m = 2n$) and the exchange matrix at $t_{0}$ is the \emph{principal matrix} corresponding to $B_{0}$ given by  
\begin{equation}
\tilde{B}_{t_{0}} = \left( \begin{array}{c} B^{0} \\ I
                             \end{array} \right)
\end{equation}
where $I$ is the $n \times n$ identity matrix.  Denote the corresponding cluster algebra by $\mathcal{A}_{\bullet} = \mathcal{A}_{\bullet}(B^{0}, t_{0})$. 
\end{definition}

\begin{definition} Let $\mathcal{A}_{\bullet} = \prinA(B^{0}, t_{0})$ be the cluster algebra with principal coefficients, with labeled seed at $t_{0}$ written as 
\begin{eqnarray}  \label{eq:initial-seed}
\tilde{\xx}_{t_{0}} = (x_{1}, \ldots, x_{n}, y_{1}, \ldots, y_{n}) \\
B_{t_{0}} = (b^0_{ij})
\end{eqnarray}
Let $\QQ_{\mbox{sf}}(z_{1}, \ldots, z_{\ell})$ denote the set of all rational functions in $\ell$ independent variables $z_{1}, \ldots, z_{\ell}$ which can be expressed as subtraction-free rational expression in $z_{1}, \ldots, z_{\ell}$.
Observe that by iterating the exchange relations in (\ref{eq:exchange-relation}), any cluster variable $x_{\ell; t} \in \prinA$ can be expressed as a unique rational function in $x_{1}, \ldots, x_{n}, y_{1}, \ldots, y_{n}$ given as a subtraction-free rational expression.  Denote this rational function by 
\begin{eqnarray}
X_{\ell; t} = X^{B^{0}; t_{0}}_{\ell; t} \in \Qsf(x_{1}, \ldots, x_{n}, y_{1}, \ldots, y_{n}).
\end{eqnarray}
Let $F_{\ell; t} = F^{B^{0}; t_{0}}_{\ell; t} \in \QQ_{sf}(u_{1}, \ldots, u_{n})$ denote the rational expression obtained by setting $x_{i} = 1$ and $y_{i} = u_{i}$ for all $i$ in $X_{\ell; t}$:
\begin{eqnarray}
F_{\ell; t} = X_{\ell; t}(1, \ldots, 1, u_{1}, \ldots, u_{n}). 
\end{eqnarray}
\end{definition}
Using \cite[Proposition~11.2]{ca2}, which is a sharpened version of the "Laurent phenomenon" (see \cite[Theorem~3.1]{ca1}) for cluster algebras with principal coefficients, the next theorem shows that the $F_{\ell; t}$ functions are indeed polynomials:

\begin{theorem} Let $\prinA = \prinA(B^{0}, t_{0})$ be a cluster algebra with principal coefficients at $t_{0}$.  Suppose that the initial seed in $\prinA$ is given as at (\ref{eq:initial-seed}).  Then 
\begin{eqnarray}
X_{\ell; t} & \in & \ZZ[x_{1}^{\pm 1}, \ldots, x_{n}^{\pm 1}, y_{1}, \ldots, y_{n}],  \\
F_{\ell; t} & \in  & \ZZ[u_{1}, \ldots, u_{n}].
\end{eqnarray}
\end{theorem}

The $F$-polynomials may be computed using the following recurrence.

\begin{proposition}
\label{fpoly-rec}
Let $t \mapsto \tilde B_t=(b^t_{ij}) \, (t\in\TT_n)$ be the family of
$2n\times n$ matrices associated with the cluster algebra $\prinA(B^{0}, t_{0})$.  Then the polynomials
$F_{\ell;t} = F_{\ell;t}^{B_{t_0};t_0}(u_1,\dots,u_n)$
are uniquely determined by the initial conditions
\begin{equation}
\label{fpoly-init}
F_{\ell;t_0} = 1 \quad (\ell = 1,
\dots, n) \ ,
\end{equation}
together with the recurrence relations
\begin{align}
\label{fpoly-rec1-nq}
F_{\ell;t'}&= F_{\ell;t} \quad \text{for $\ell\neq
  k$;}\\
\label{fpoly-rec2}
F_{k;t'} &= \frac{
\prod_{i=1}^n F_{i;t}^{[b_{ik}^t]_+}\prod_{j=1}^n u_j^{[b_{n+j,k}^t]_+}
+
\prod_{i=1}^n F_{i;t}^{[-b_{ik}^t]_+}\prod_{j=1}^n u_j^{[-b_{n+j,k}^t]_+}}{F_{k; t}}\,,
\end{align}
for every edge $t \overunder{k}{} t'$ such that $t$ lies on
the (unique) path from $t_0$ to $t'$ in~$\TT_n\,$.
\end{proposition}

In \cite[Section 6]{coefficients}, the following $\ZZ^{n}$-grading of $\prinA = \prinA(B^{0}, t_{0})$ was introduced:

\begin{proposition} \label{def:g-vecs} Define a $\ZZ^{n}$-grading of $\ZZ[x_{1}^{\pm 1}, \ldots, x_{n}^{\pm 1}, y_{1}^{\pm}, \ldots, y_{n}^{\pm}]$ by
\begin{eqnarray} 
\deg(x_{i}) = \ee_{i},  \hspace{1cm} \deg(y_{j}) = -\bb^{0}_{j},
\end{eqnarray}
where $\ee_{1}, \ldots, \ee_{n}$ are the standard basis vectors for $\ZZ^{n}$ and $\bb^{0}_{j}$ is the $j$th column of $B^{0}$.  Under this $\ZZ^{n}$-grading, the cluster algebra $\prinA(B^{0}, t_{0})$ is a $\ZZ^{n}$-graded subalgebra of $\ZZ[x_{1}^{\pm 1}, \ldots, x_{n}^{\pm 1}, y_{1}^{\pm}, \ldots, y_{n}^{\pm}]$, and every cluster variable in $\prinA(B^{0}, t_{0})$ is a homogeneous element.
\end{proposition}
The $\ZZ^{n}$-degree of the cluster variable $x_{\ell; t} \in \prinA(B^{0}, t_{0})$ will be denoted by $\gg_{\ell; t} = \gg_{\ell; t}^{B^{0}; t} = \left[ \begin{array}{c} g_{1} \\ \vdots \\ g_{n} \end{array} \right] \in \ZZ^{n}$; we refer to $\gg_{\ell; t}$ as the \emph{$\gg$-vector} of $x_{\ell; t}$.  

For any cluster algebra $\mathcal{A}$ such that the initial extended cluster at $t_{0}$ is given by $\tilde{\xx} = (x_{1}, \ldots, x_{m})$ and the exchange matrix at $t_{0}$ is $\tilde{B}^{0} = (b^{0}_{ij})$, let  

\begin{eqnarray} \label{def:yhat}
y_{j} = \prod_{i = n + 1}^{m} x_{i}^{b^{0}_{ij}},  \hspace{1cm} \hat{y}_{j} = \prod_{i = 1}^{m} x_{i}^{b^{0}_{ij}} \hspace{0.5cm}
\end{eqnarray}
for $j \in [1, n]$.

The next theorem shows that any cluster variable in $\mathcal{A}$ may be computed in terms of the initial extended cluster if the corresponding $F$-polynomial and $\gg$-vector are known.

\begin{theorem} \label{thm:cluster-var-formula} Let $\mathcal{A}$ be a cluster algebra such that the principal part of the exchange matrix at $t_{0}$ is $B^{0}$.  Using the notation just given for the initial seed at $\mathcal{A}$, any cluster variable $x_{\ell; t}$ in $\mathcal{A}$ may be expressed as 
\begin{eqnarray}
x_{\ell; t} = \displaystyle \frac{F^{B^{0}; t_{0}}_{\ell; t}(\hat{y}_{1}, \ldots, \hat{y}_{n})}{F^{B^{0}; t_{0}}_{\ell; t}|_{\mathbb{P}}(y_{1}, \ldots, y_{n})} x_{1}^{g_{1}} \cdots x_{n}^{g_{n}},
\end{eqnarray}
where $\gg_{\ell; t}^{B^{0}; t_{0}} = \left[ \begin{array}{c} g_{1} \\ \vdots \\ g_{n} \end{array} \right]$ and $\mathbb{P} = \Trop(x_{n + 1}, \ldots, x_{m})$. 

Furthermore, for $j \in [1, n]$, $u_{j}$ does not divide $F_{\ell; t}$.  Thus, in the cluster algebra $\prinA = \prinA(B^{0}, t_{0})$, 
\begin{eqnarray} \label{eq:cluster-var-fpoly}
x_{\ell; t} = F^{B^{0}; t_{0}}_{\ell; t}(\hat{y}_{1}, \ldots, \hat{y}_{n})x_{1}^{g_{1}}\cdots x_{n}^{g_{n}}.
\end{eqnarray}
\end{theorem}

Next, we introduce \emph{extended $\gg$-vectors} which generalize the $\gg$-vectors to any cluster algebra $\mathcal{A}$. 

\begin{definition} \label{def:extended-g-vec} Let $\mathcal{A}$ be any cluster algebra with initial extended cluster $\tilde{\xx} = (x_{1}, \ldots, x_{m})$, and let $\ell \in [1, n]$, $t \in \mathbb{T}_{n}$.   
Define $\gg_{\ell; t}^{\tilde{B}^{0}; t_{0}} = (g_{1}, \ldots, g_{m})$ to be the unique vector in $\ZZ^{m}$ such that the cluster variable $x_{\ell; t} \in \mathcal{A}$ satisfies
\begin{eqnarray} 
x_{\ell; t} = F_{\ell; t}^{\tilde{B}^{0}; t_{0}}(\hat{y}_{1}, \ldots, \hat{y}_{n})x_{1}^{g_{1}}\ldots x_{m}^{g_{m}}.
\end{eqnarray}
For $\ell \in [n + 1, m]$, define $\gg_{\ell; t}^{\tilde{B}^{0}; t_{0}} = \ee_{\ell} \in \ZZ^{m}$.  We call $\tilde{\gg}_{\ell; t} = \tilde{\gg}_{\ell; t}^{\tilde{B}^{0}; t_{0}}$ the \emph{extended $\gg$-vector}.  
\end{definition}

 Observe that  the existence of $\tilde{\gg}_{\ell; t}$ follows immediately from Theorem \ref{thm:cluster-var-formula}.   Explicitly, the first $n$ coordinates of $\tilde{\gg}_{\ell; t}$ are given by $\gg_{\ell; t}^{B^{0}; t_{0}}$, and the remaining $m - n$ coordinates are given by the exponents of $x_{n + 1}, \ldots, x_{m}$ in the expression 
\begin{eqnarray}
\frac{1}{F_{\ell; t}^{B^{0}; t_{0}}|_{\PP} (y_{1}, \ldots, y_{n})},
\end{eqnarray}
where $\PP = \Trop(x_{n + 1}, \ldots, x_{m}).$  In particular, $\gg_{\ell; t}^{B^{0}; t_{0}} = \tilde{\gg}_{\ell; t}^{\tilde{B}^{0}; t_{0}}$ when $\tilde{B}^{0}$ is principal.  (Here, we extend $\gg_{\ell; t}^{B^{0}; t_{0}} \in \ZZ^{n}$ to an element of $\ZZ^{2n}$ by appending $n$ 0's.)

The next proposition gives a recurrence relation for the extended $\gg$-vectors.    Let $\tilde{b}^{j}$ denote the $j$th column of the matrix $\tilde{B}^{0}$ $(j \in [1, n])$.  Let $\tilde{B}^{t} = (b^{t}_{ij})$ denote the matrix obtained from $\tilde{B}^{0}$ by mutating the matrix from $t_{0}$ to vertex $t \in \mathbb{T}_{n}$, and let $(b^{\bullet, t}_{ij})$ be the $2n \times n$ matrix obtained by mutating the principal matrix corresponding to $B^{0}$ from $t_{0}$ to $t$.  

\begin{proposition} \label{h-vec-props}
For $\ell \in [1, m]$, $t \in \mathbb{T}_{n}$, $\tilde{\gg}_{\ell; t}$ is given by the initial conditions
\begin{equation}
\label{hvec-init}
\tilde{\gg}_{\ell;t_{0}} = \ee_\ell \quad (\ell = 1,
\dots, m)
\end{equation}
together with the recurrence relations
\begin{eqnarray}
\label{hvec-rec1}
\tilde{\gg}_{\ell;t'} & = & \tilde{\gg}_{\ell;t} \quad \text{for $\ell\neq
  k$;}\\
\label{hvec-rec2}
\tilde{\gg}_{k;t'} & = & -\tilde{\gg}_{k;t} +\sum_{i=1}^m [-b^t_{ik}]_+ \tilde{\gg}_{i;t}
-\sum_{j=1}^n  [-b^{\bullet, t}_{n+j,k}]_+ \tilde{\textbf{b}}^{j}\,
 \,,
\end{eqnarray}
where  the (unique) path from $t_0$ to $t'$ in $\TT_n$ ends
with the edge $t \overunder{k}{} t'$. 
\end{proposition}

\begin{proof}    Equation (\ref{hvec-init}) is clear from the fact that $F_{\ell; t_{0}} = 1$ for all $\ell \in [1, n]$. 
Let $t \frac{k}{\hspace{1cm}} t'$ in $\mathbb{T}_{n}$.  If $\ell \neq k$, then $F_{\ell; t} = F_{\ell; t'}$, and (\ref{hvec-rec1}) follows.  Now consider $\tilde{\gg}_{k; t'}$.

 In section 3 of \cite{coefficients}, an element $Y_{k; t} \in \QQ_{\mbox{sf}}(u_{1}, \ldots, u_{n})$ was defined satisfying certain properties that we recall below.    By \cite[Proposition 3.12]{coefficients},
\begin{eqnarray}
\label{eq:Y-F-2}
F_{k;t'} = \frac{Y_{k;t} + 1}
{(Y_{k;t} + 1)|_{\Trop(u_1, \dots, u_n)}} \cdot
F_{k;t}^{-1} \prod_{i=1}^n F_{i;t}^{[-b^{t}_{ik}]_+} \ ,
\end{eqnarray}
where the right hand side is computed in the field $\QQ(u_{1}, \ldots, u_{n})$ of rational functions.  By \cite[(3.16)]{coefficients}, the cluster variable $x_{k; t'} \in \mathcal{A}$ is given by 

\begin{eqnarray}
\label{eq:x-X}
x_{k;t'} = \frac{(Y_{k;t} + 1)|_\Fcal(\hat y_1, \dots, \hat y_n)}
{(Y_{k;t} + 1)|_\PP(y_1, \dots, y_n)} \cdot
x_{k;t}^{-1} \prod_{i=1}^n x_{i;t}^{[-b_{ik}^{t}]_+} \ ,
\end{eqnarray} 
where $\PP = \Trop(x_{n + 1}, \dots, x_m)$, and $\hat{y}_{i}$, $y_{i}$ are elements of the ambient field $\mathcal{F}$ as defined at (\ref{def:yhat}).

By \cite[(3.14)]{coefficients}, 
\begin{eqnarray}
\label{eq:yY-geometric}
Y_{k;t}|_{\PP}(y_1, \dots, y_n) = \prod_{i=n+1}^m x_i^{b^t_{ik}} \in \mathcal{F}.
\end{eqnarray}
 
 By (\ref{eq:yY-geometric}), 
\begin{eqnarray}
(Y_{k; t} + 1)|_{\PP}(y_{1}, \ldots, y_{n}) = \prod_{i = n + 1}^{m} x_{i}^{\min(0, b_{ik}^{t})}.
\end{eqnarray} 
 
Since $\min(0, b_{ik}^{t}) = -[-b^{t}_{ik}]_{+}$, equation (\ref{eq:x-X}) may be rewritten as

\begin{eqnarray} \label{eqn:Y1}
x_{k; t'} = (Y_{k; t} + 1)(\hat{y}_{1}, \ldots, \hat{y}_{n}) x_{k; t}^{-1} \prod_{i = 1}^{m}x_{i; t}^{[-b_{ik}^{t}]_{+}},
\end{eqnarray}

By Definition \ref{def:extended-g-vec}, it follows from equation (\ref{eqn:Y1}) that 
\begin{eqnarray}
x_{k; t'} = P_{k; t}(\hat{y}_{1}, \ldots, \hat{y}_{n})x_{1}^{g_{1}'}\ldots x_{m}^{g_{m}'},
\end{eqnarray}
where 
\begin{eqnarray}
P_{k; t} = (Y_{k; t} + 1) F_{k; t}^{-1} \prod_{i = 1}^{n} F_{i; t}^{[-b_{ik}^{t}]_{+}}, \\ 
(g_{1}', \ldots, g_{m}') = -\tilde{\gg}_{k;t} +\sum_{i=1}^m [-b^t_{ik}]_+ \tilde{\gg}_{i;t}.
\end{eqnarray}

In the particular case, where $\mathcal{A}$ has principal coefficients, (\ref{eq:yY-geometric}) implies that 
\begin{eqnarray}
(Y_{k; t} + 1)|_{\Trop(y_{1}, \ldots, y_{n})}(y_{1}, \ldots, y_{n}) = \prod_{j = 1}^{n} y_{j}^{\min(0, b_{j + n, k}^{\bullet, t})}
\end{eqnarray}

Thus, 
\begin{eqnarray}
(Y_{k; t} + 1)|_{\Trop(u_{1}, \ldots, u_{n})} = \prod_{j = 1}^{n} u_{j}^{\min(0, b_{j + n, k}^{\bullet, t})} \in \QQ_{\mbox{sf}}(u_{1}, \ldots, u_{n}).
\end{eqnarray}

Equation (\ref{eq:Y-F-2}) may be rewritten as 
\begin{eqnarray}
\label{eqn:FY-formula} F_{k; t'} = (Y_{k; t} + 1) F_{k; t}^{-1} \prod_{i = 1}^{n} F_{i; t}^{[-b_{ik}^{t}]_{+}} \prod_{j = 1}^{n} u_{j}^{[-b_{j + n, k}^{\bullet, t}]_{+}} 
\end{eqnarray}

It follows that
\begin{eqnarray}
x_{k; t'} = F_{k; t'}(\hat{y}_{1}, \ldots, \hat{y}_{n})x_{1}^{g_{1}'}\ldots x_{m}^{g_{m}'} \prod_{j = 1}^{n} \hat{y}_{j}^{-[-b_{j + n, k}^{\bullet, t}]_{+}}, 
\end{eqnarray}
and equation (\ref{hvec-rec2}) follows from Definition \ref{def:extended-g-vec}.

\end{proof}

\begin{proposition} [{\cite[Proposition~6.6]{coefficients}}] \label{g-vec-rec-b} In the particular case when $\tilde{B}^{0}$ is principal and $\gg_{i; t}^{B^{0}; t_{0}} = \tilde{\gg}_{i; t}^{\tilde{B}^{0}; t_{0}}$ for all $i \in [1, n], t \in \mathbb{T}_{n}$, the recurrence (\ref{hvec-rec2}) may be replaced by the following recurrence:
\begin{eqnarray}
\label{hvec-rec2b}
\gg_{k;t'} & = & -\gg_{k;t} +\sum_{i=1}^{2n} [\epsilon b^t_{ik}]_+ \gg_{i;t}
-\sum_{j=1}^n  [\epsilon b^{\bullet, t}_{n+j,k}]_+ \tilde{\textbf{b}}^{j}\,
 \,,
\end{eqnarray}
for $\epsilon \in \{ +, - \}$. 
\end{proposition}

\section{Quantum Cluster Algebras} \label{section:quantum}
In this section, we define quantum cluster algebras and state properties to be used later.  See \cite[Sections 3 and 4]{quantum} for additional information.

Let $L$ be a lattice of rank $m$, with associated skew-symmetric bilinear form $\Lambda: L \times L \rightarrow \ZZ$.  We also introduce a formal variable $q$ which commutes with all other elements, and consider the ring $\ZZ[q^{\pm \frac{1}{2}}]$ of integer Laurent polynomials in the variable $q^{\frac{1}{2}}$.

\begin{definition} \cite[Definition 4.1]{quantum} \label{def:torus} The \emph{based quantum torus} associated with $L$ is the $\mathbb{Z}[q^{\pm \frac{1}{2}}]$-algebra $\mathcal{T}$ with distinguished $\mathbb{Z}[q^{\pm \frac{1}{2}}]$-basis $\{ X^{\ee} : \ee \in L  \}$ and multiplication given by 
\begin{equation}
X^{\ee}X^{\ff} = q^{\Lambda(\ee, \ff)/2}X^{\ee + \ff} \hspace{1cm} (\ee, \ff \in L).
\end{equation}
\end{definition}
It is easy to check that $\mathcal{T}$ is associative.  The basis elements satisfy the following commutation relations:
\begin{equation}
X^{\ee}X^{\ff} = q^{\Lambda(\ee, \ff)}X^{\ff}X^{\ee}.
\end{equation} 
In addition, we have
\begin{equation}
X^{0} = 1, \hspace{0.75cm} (X^{\ee})^{-1} = X^{-\ee} \hspace{0.75cm} (\ee \in L).
\end{equation}
The based quantum torus $\mathcal{T}$ is an Ore domain, which means that it is contained in $\mathcal{F}$, its skew-field of fractions.  The quantum cluster algebra to be defined below will be a $\ZZ[q^{\pm \frac{1}{2}}]$-subalgebra of $\mathcal{F}$.

Our first goal is to define the quantum analogue of a labeled seed.  First, we must make some definitions.

\begin{definition} \cite[Definition 4.3]{quantum} A \emph{toric frame} in $\mathcal{F}$ is a mapping $M: \ZZ^{m} \rightarrow \mathcal{F} - \{0\}$ of the form 
\begin{equation}
M(\cc) = \phi(X^{\eta(\cc)}),
\end{equation}
where $\phi$ is an automorphism of $\mathcal{F}$, and $\eta: \ZZ^{m} \rightarrow L$ is an isomorphism of lattices.
\end{definition}
By definition, the elements $M(\cc)$ form a basis of $\phi(\mathcal{T})$, which is an isomorphic copy of $\mathcal{T}$ in $\mathcal{F}$.  The commutation relations of these elements are given by 
\begin{eqnarray} \label{eq:toric-product}
M(\cc)M(\dd) & = & q^{\Lambda_{M}(\cc, \dd)/2}M(\cc + \dd), \\
M(\cc)M(\dd) & = & q^{\Lambda_{M}(\cc, \dd)}M(\dd)M(\cc), 
\end{eqnarray}
where the bilinear form $\Lambda_{M}$ on $\ZZ^{m}$ is obtained by transferring the form $\Lambda$ from $L$ via the lattice isomorphism $\eta$.  We also have 
\begin{equation}
M(0) = 1, \hspace{0.5cm} M(\cc)^{-1} = M(-\cc) \hspace{0.5cm} (\cc \in \ZZ^{m}).
\end{equation}
We use the same symbol $\Lambda_{M}$ to denote the skew-symmetric matrix whose entries are given by 
\begin{eqnarray} \label{eq:lambda-entries}
\lambda_{ij} = \Lambda_{M}(\ee_{i}, \ee_{j})
\end{eqnarray}
where $\ee_{1}, \ldots, \ee_{m}$ are the standard basis vectors of $\ZZ^{m}$.  For a given toric frame, write $X_{i} = M(e_{i})$ for $i \in [1, m]$.  The elements $X_{i}$ satisfy the \emph{quasi-commutation relations}:
\begin{equation} \label{eq:quasi-commutation}
X_{i}X_{j} = q^{\lambda_{ij}}X_{j}X_{i}.
\end{equation}
Using (\ref{eq:toric-product}) and (\ref{eq:lambda-entries}), it follows that for any $\cc = (c_{1}, \ldots, c_{m}) \in \ZZ^{m}$, 
\begin{eqnarray} \label{def:monomial}
M(\cc) = q^{\frac{1}{2} \sum_{\ell < k} c_{k}c_{\ell}\lambda_{k\ell}}X_{1}^{c_{1}}\ldots X_{m}^{c_{m}},
\end{eqnarray}
so that the toric frame $M$ is uniquely determined by the elements $X_{1}, \ldots, X_{m}$.

\begin{definition} \cite[Definition 3.1]{quantum}  \label{def:compatible} Let $\tilde{B} = (b_{ij})$ be an $m \times n$ integer matrix, and let $\Lambda = (\lambda_{ij})$ be a skew-symmetric $m \times m$ integer matrix.  We say that the pair $(\Lambda, \tilde{B})$ is \emph{compatible} if for every $j \in [1, n]$ and $i \in [1, m]$, we have 
\begin{eqnarray}
\displaystyle \sum_{k = 1}^{m} b_{kj}\lambda_{ki} = \delta_{ij}d_{j}
\end{eqnarray}
for some positive integers $d_{j}$ $(j \in [1, n])$.  In other words, the product $\tilde{B}^{T}\Lambda$ is equal to the $n \times m$ matrix $( D | 0 )$, where $D$ is a $n \times n$ diagonal matrix with positive integer diagonal entries $d_{1}, \ldots, d_{n}$. 
\end{definition}

\begin{proposition} \cite[Proposition 3.3]{quantum} \label{prop:compatible}
If a pair $(\Lambda, \tilde{B})$ is compatible, then the principal part of  $\tilde{B}$ is skew-symmetrizable.
\end{proposition}

\begin{definition} \cite[Definition 4.5]{quantum} \label{def:quantum-seed} A \emph{quantum seed} is a pair $(M, \tilde{B})$, where 
\begin{itemize}
  \item $M$ is a toric frame in $\mathcal{F}$, and 
  \item $\tilde{B}$ is an $m \times n$ integer matrix such that $(\Lambda_{M}, \tilde{B})$ is a compatible pair in the sense of Definition \ref{def:compatible}.
\end{itemize}
\end{definition}
Suppose that in the classical limit (i.e., taking $q = 1$), the elements $X_{1}, \ldots, X_{m}$ specialize to $x_{1}, \ldots, x_{m}$.  Then these elements $x_{1}, \ldots, x_{m}$ form a free generating set of the ambient field, so by Proposition \ref{prop:compatible}, we have that $(\tilde{\xx} = (x_{1}, \ldots, x_{m}), \tilde{B})$ is a labeled seed.

The next goal is to define mutation for quantum seeds.  First, we extend the definition of matrix mutations to compatible pairs.  Let $k \in [1, n]$, and pick a sign $\epsilon \in \{ \pm 1 \}$.  Denote by $E_{\epsilon}$ the $m \times m$ matrix with entries given by 
\begin{equation} \label{eq:E-epsilon}
e_{ij} = 
\begin{cases}
  \delta_{ij} & \text{if } j \neq k \\
  -1 & \text{if } i = j = k \\
  \max(0, -\epsilon b_{ik}) & \text{if } i \neq j = k. 
\end{cases}
\end{equation}
Set 
\begin{eqnarray} \label{eq:Lambda-mutation}
\Lambda' = E_{\epsilon}^{T}\Lambda E_{\epsilon}.
\end{eqnarray}
Then $\Lambda'$ is skew-symmetric.  Furthermore, one may verify that $\Lambda'$ is independent of sign $\epsilon$ and that $(\Lambda', \mu_{k}(\tilde{B}))$ is a compatible pair (see \cite[Proposition 3.4]{quantum}).  

\begin{definition} \cite[Definition 3.5]{quantum} Let $(\Lambda, \tilde{B})$ be a compatible pair, and let $k \in [1, n]$.  We say that $(\Lambda', \mu_{k}(\tilde{B}))$ (with $\Lambda'$ as above) is obtained by \emph{mutation} in direction $k$ from $(\Lambda, \tilde{B})$, and write $\mu_{k}(\Lambda, \tilde{B}) = (\Lambda', \mu_{k}(\tilde{B}))$. 
\end{definition}

Write
\begin{equation}
\left(\hspace{-0.1cm}\begin{array}{c}r \\ p\end{array}\hspace{-0.1cm}\right)_{t} = \displaystyle \frac{(t^{r} - t^{-r})\cdots(t^{r-p+1} - t^{-r+p-1})}{(t^{p} - t^{-p})\cdots(t - t^{-1}) }
\end{equation}
for the $t$-binomial coefficient.   The $t$-binomial coefficients satisfy the equation
\begin{eqnarray} \label{eqn:t-binomial-formula}
\prod_{p = 0}^{r - 1} (1 + t^{r - 1 - 2p}x) = \sum_{p = 0}^{r} \left(\hspace{-0.1cm}\begin{array}{c}r \\ p\end{array}\hspace{-0.1cm}\right)_{t} x^p
\end{eqnarray}

For mutations of toric frames, let $(M, \tilde{B})$ be a quantum seed, with $\tilde{B} = (b_{ij})$.  Fix an index $k \in [1, n]$ and a sign $\epsilon \in \{ \pm 1 \}$.  Define a mapping $M': \ZZ^{m} \rightarrow \mathcal{F} - \{0\}$ by setting, for $\cc = (c_{1}, \ldots, c_{n}) \in \ZZ^{m}$ such that $c_{k} \geq 0$,
\begin{equation} \label{eq:toric-frame-mutation}
M'(\cc) = \displaystyle \sum_{p = 0}^{c_{k}} \left(\hspace{-0.1cm}\begin{array}{c}c_{k}\\p  \end{array}\hspace{-0.1cm}\right)_{q^{d_{k}/2}} M(E_{\epsilon}\cc + \epsilon p \bb^{k}), \hspace{0.5cm} M'(-\cc) = M'(\cc)^{-1},
\end{equation} 
where $\bb^{k}$ denotes the $k$th column of $\tilde{B}$, and the matrix $E_{\epsilon}$ is given at (\ref{eq:E-epsilon}).

\begin{proposition} \cite[Proposition 4.7]{quantum}
\begin{enumerate}
\item[(1)] The mapping $M'$ is a toric frame which does not depend on the choice of sign $\epsilon$.
\item[(2)] The pair $(\Lambda_{M'}, \mu_{k}(\tilde{B}))$ is compatible and is obtained from $(\Lambda_{M}, \tilde{B})$ by mutation in direction $k$.
\item[(3)] $(M', \mu_{k}(\tilde{B}))$ is a quantum seed.
\end{enumerate}
\end{proposition}

This proposition justifies the next definition: 

\begin{definition} \cite[Definition 4.8]{quantum} Let $(M, \tilde{B})$ be a quantum seed, and let $k \in [1, n]$.  Suppose that $M'$ is given at (\ref{eq:toric-frame-mutation}) and $\tilde{B}' = \mu_{k}(\tilde{B})$.  Then we say that the quantum seed $(M', \tilde{B}')$ is obtained from $(M, \tilde{B})$ by \emph{mutation} in direction $k$, and write $(M', \tilde{B}') = \mu_{k}(M, \tilde{B})$.
\end{definition}

\emph{Quantum cluster patterns} may be defined in a completely analogous way to cluster patterns by simply replacing the labeled seeds $(\xx_{t}, \tilde{B}_{t})$ by the quantum seeds $(M_{t}, \tilde{B}_{t})$ in Definition \ref{def:cluster-pattern}.  In this case, we write $X_{j; t} = M_{t}(\ee_{j})$ for $j \in [1, m]$, $t \in \mathbb{T}_{n}$.  The \emph{cluster variables} are the elements $X_{j; t}$ for $j \in [1, n]$, $t \in \mathbb{T}_{n}$.  Observe that $X_{j; t} = X_{j; t_{0}}$ for $i \in [n + 1, m]$, $j \in \mathbb{T}_{n}$; these are the $m - n$ \emph{coefficient variables} which do not depend on the seed $t$.  The \emph{cluster} at the seed $t$ is $(X_{1; t}, \ldots, X_{n; t})$, and the \emph{extended cluster} is $(X_{1; t}, \ldots, X_{m; t})$.

The next proposition provides the analogue in quantum cluster algebras of the exchange relation given at (\ref{eq:exchange-relation}).

\begin{proposition} \cite[Proposition 4.9]{quantum}  \label{prop:quantum-exchrel} Let $(M, \tilde{B})$ be a quantum seed, and suppose the quantum seed $(M', \tilde{B}')$ is obtained from $(M, \tilde{B})$ by mutation in direction $k \in [1, n]$.  For $i \in [1, n]$, set $X_{i} = M(\ee_{i})$, $X'_{i} = M'(\ee_{i})$.  Then $X'_{i} = X_{i}$ for $i \neq k$, and 
\begin{equation}
X'_{k} = M(-\ee_{k} + \displaystyle \sum_{i \in [1, m]} [b_{ik}]_{+} \ee_{i}) + M(-\ee_{k} + \displaystyle \sum_{i \in [1, m]} [-b_{ik}]_{+} \ee_{i}).
\end{equation}
\end{proposition}
Finally, we are ready for the definition of quantum cluster algebra.  Observe that one consequence of Proposition \ref{prop:quantum-exchrel} is that in a given quantum cluster pattern, $X_{j; t} = X_{j; t'}$ for any $j \in [n + 1, m]$, $t, t' \in \mathbb{T}_{n}$.  Put $X_{j} = X_{j; t}$ for $j \in [n + 1, m]$, and write $\mathcal{X} = \{ X_{j; t} : j \in [1, n], t \in \mathbb{T}_{n} \}$ for the collection of all cluster variables.

\begin{definition} \cite[Definition 4.12]{quantum} \label{def:quantum-cluster-alg} For a given  quantum cluster pattern $t \mapsto (M_{t}, \tilde{B}_{t})$, the associated \emph{quantum cluster algebra} $\mathcal{A}$ is the $\ZZ[q^{\pm \frac{1}{2}}, X_{n + 1}^{\pm 1}, \ldots, X_{m}^{\pm 1}]$-subalgebra of the ambient skew-field $\mathcal{F}$, generated by the elements of $\mathcal{X}$.
\end{definition}

The next result will be referred to as the \textit{quantum Laurent phenomenon}:

\begin{theorem} \cite[Corollary 5.2]{quantum} \label{thm:quantum-laurent} The cluster algebra $\mathcal{A}$ as above is contained in $\ZZ[q^{\pm \frac{1}{2}}, X_{1}^{\pm 1}, \ldots, X_{m}^{\pm 1}]$.
\end{theorem}

\section{F-polynomials in Quantum Cluster Algebras} \label{section:quantum-f-polys}
\noindent For this section, fix the following:
\begin{itemize}
  \item a vertex $t_{0} \in \mathbb{T}_{n}$;
  \item an $n \times n$ skew-symmetrizable integer matrix $B^{0} = (b^{0}_{ij})$;
  \item an $n \times n$ diagonal matrix $D$ with positive integer entries $d_{1}, \ldots, d_{n}$ on the diagonal satisfying the property that $DB^{0}$ is skew-symmetric.
\end{itemize}
In this section and the next, we will assume all quantum cluster algebras $\mathcal{A}$ under consideration have the following properties.  The extended exchange matrix at $t_{0}$, which will be denoted by $\tilde{B}^{0}$, has principal part $B^{0}$.   The $m \times m$ skew-symmetric integer matrix  which gives the quasi-commutation relations for the cluster variables and coefficients at $t_{0}$ will be denoted by $\Lambda_{0}$, and this matrix will be assumed to satisfy the following compatibility condition with $\tilde{B}^{0}$:
\begin{eqnarray} \label{compatibility}
(\tilde{B}^{0})^{T} \Lambda_{0} = ( D | 0 ).
\end{eqnarray}

For any such quantum cluster algebra $\mathcal{A}$, denote the ambient skew-field of the quantum cluster algebra $\mathcal{A}$ by $\mathcal{F}$.   Let $\tilde{B}^{t} = (b^{t}_{ij})$ denote the extended exchange matrix at the vertex $t$.  Write $\tilde{\bb}^{j; t} \in \mathbb{Z}^m$ for the $j$th column of $\tilde{B}^{t}$, and set $\tilde{\bb}^{j} = \tilde{\bb}^{j; t_{0}}$.  Let $M_{t}$ denote the toric frame at the vertex $t \in \mathbb{T}_{n}$, and put $M_{0} := M_{t_{0}}$.  Let $\Lambda_{t}$ denote the skew-symmetric bilinear form on $\mathbb{Z}^{m}$ corresponding to $M_{t}$.  (We will also use $\Lambda_{t}$ for the $m \times m$ skew-symmetric integer matrix which gives the quasi-commutation relations in the cluster at $t$).  Denote the cluster variables in the corresponding cluster by $X_{j; t} := M_{t}(e_{j})$ for $1 \leq j \leq n$, and write $X_{j} = X_{j; t_{0}}$ for the cluster variables in the initial cluster.   For $j \in [1, n]$, $t \in \mathbb{T}_{n}$, write
\begin{eqnarray} \label{def:y-hat}
\hat{Y}_{j; t}  =  M_{t}(\tilde{\bb}^{j; t}), \hspace{0.5cm} \hat{Y}_{j} =  \hat{Y}_{j; t_{0}}
\end{eqnarray}
From Remark 4.6 of \cite{quantum}, it is known that 

\begin{eqnarray} \label{yhat-comm2}
\Lambda_{t}(\tilde{\bb}^{i; t}, \tilde{\bb}^{j; t}) = d_{i}b^{t}_{ij}.
\end{eqnarray}
Equivalently, the elements $\hat{Y}_{j; t}$ satisfy the quasi-commutation relation
\begin{eqnarray} \label{yhat-comm1}
\hat{Y}_{i; t} \hat{Y}_{j; t} = q^{d_{i}b^{t}_{ij}} \hat{Y}_{j; t} \hat{Y}_{i; t}.
\end{eqnarray}
\noindent Note that these quasi-commutation relations are the same in every quantum cluster algebra under consideration since they only depend on the entries of the matrices $D$ and $B^{0}$.

Let $R = R_{B^{0}, D}$ be the $\mathbb{Z}[q^{\pm \frac{1}{2}}]$-algebra with generating set \{$Z_{j} : 1 \leq j \leq n \}$ such that the $Z_{i}$ satisfy the same quasi-commutation relations as the $\hat{Y}_{i}$:
\begin{eqnarray}
Z_{i}Z_{j} = q^{d_{i}b^{0}_{ij}}Z_{j}Z_{i}.
\end{eqnarray}

Then $R$ is a based quantum torus in the sense of Definition \ref{def:torus}.  Thus, we can consider the skew-field $\mathcal{F}(R)$ of right fractions of $R$:
\begin{eqnarray}
\mathcal{F}(R) = \{ FG^{-1} : F, G \in R \}.
\end{eqnarray}
For $F \in \mathcal{F}(R)$,  denote by $F(\hat{Y})$ the element of $\mathcal{F}$ obtained from $F$ by setting each $Z_{i}$ to $\hat{Y}_{i}$ for all $1 \leq i \leq n$.    In analogy to the notation at (\ref{def:monomial}), for $\textbf{c} = (c_{1}, \ldots, c_{n}) \in \ZZ^{n}$, define 
\begin{eqnarray} \label{eqn:Z-monomial}
Z^{\textbf{c}} = q^{\frac{1}{2} \sum_{1 \leq i < j \leq n} d_{j}b^{0}_{ji}c_{i}c_{j}} Z_{1}^{c_{1}} \cdots Z_{n}^{c_{n}}.
\end{eqnarray}
     
\begin{proposition} [{Example 0.5}, {\cite{oberwolfach}}] \label{prop:prinA-quantum} Suppose $\Lambda$ is an $n \times n$ skew-symmetric integer matrix.   The cluster algebra $\mathcal{A}_{\bullet} = \mathcal{A}_{\bullet}(B^{0}, t_{0})$ has a unique quantization such that the quasi-commutation relations of the cluster variables in the cluster at $t_{0}$ are given by $\Lambda$.  The unique $2n \times 2n$ matrix $\Lambda_{0}$ satisfying (\ref{compatibility}) is given by 
\begin{eqnarray}
\Lambda_{0} = \left(    \begin{array}{ccc} 
                       \Lambda & &  -\Lambda B^{0} - D \\
                       -(B^{0})^{T}\Lambda + D & & (B^{0})^{T}\Lambda B^{0} + (B^{0})^{T}D  \\   \end{array}  \right).
\end{eqnarray}  

\end{proposition}

\begin{definition} We say that quantum cluster algebra defined in Proposition \ref{prop:prinA-quantum} has \emph{principal coefficients}, and denote it by $\prinA = \prinA(B^{0}, D, \Lambda, t_{0})$.
\end{definition}

\begin{theorem} \label{thm:quantum-fpoly}  Let $j \in [1, n]$, $t \in \mathbb{T}_{n}$, $\Lambda$ an $n \times n$ skew-symmetric integer matrix, and put  $\prinA = \prinA(B^{0}, D, \Lambda, t_{0})$.
Let $\gg_{j; t} = \gg_{j; t}^{B^{0}; t_{0}}$ be the $\gg$-vector (as in Proposition \ref{def:g-vecs}).
\begin{enumerate}
\item[(1)] There exists a unique polynomial $F_{j; t} = F_{j; t}^{B^{0}; D; t_{0}}$ in $Z_{1}, \ldots, Z_{n}$ with coefficients in $\mathbb{Z}[q^{\pm \frac{1}{2}}]$ such that the cluster variable $X_{j; t} \in \prinA$ is given by 
\begin{eqnarray} \label{fpoly-eqn-g}
X_{j; t} = F_{j; t}(\hat{Y})M_{0}(\gg_{j; t}),
\end{eqnarray}
where $M_{0}$ is the toric frame at $t_{0}$ for $\prinA$, and $\hat{Y}_{1}, \ldots, \hat{Y}_{n}$ are defined as at (\ref{def:y-hat}) using the columns of the principal matrix with respect to $B^{0}$.
(Here, we consider $\gg_{j; t}$ as an element in $\ZZ^{2n}$ by appending $n$ 0's to the end of the vector.)  The polynomial $F_{j; t}$ does not depend on the choice of the matrix $\Lambda$.  We call this polynomial a \emph{(left) quantum $F$-polynomial}.
\item[(2)] Let $\tilde{B}^{0}$ be an $m \times n$ integer matrix with principal part $B^{0}$.  Then there exists $\lambda_{j; t}^{\tilde{B}^{0}; t_{0}} = \lambda_{j; t} \in \frac{1}{2}\mathbb{Z}$ such that in any quantum cluster algebra $\mathcal{A}$ whose initial exchange matrix is $\tilde{B}^{0}$, we have that the cluster variable $X_{j; t} \in \mathcal{A}$ is given by 
\begin{eqnarray} \label{fpoly-eqn-h}
X_{j; t} = q^{\lambda_{j; t}}F_{j; t}(\hat{Y})M_{0}(\tilde{\gg}_{j; t}^{\tilde{B}^{0}; t_{0}}),
\end{eqnarray}
where $M_{0}$ is the toric frame at $t_{0}$ for $\mathcal{A}$, and $\hat{Y}_{1}, \ldots, \hat{Y}_{n}$ are defined using the columns of $\tilde{B}^{0}$.
(Note that $\lambda_{j; t} = \lambda_{j; t}^{\tilde{B}^{0}; t_{0}}$ depends on $\tilde{B}^{0}$ but not on  $\mathcal{A}$ or the choice of initial quasi-commutation relations.)
\item[(3)] Setting $q = 1$ and $Z_{i} = u_{i}$ ($i \in [1, n]$) in $F_{j; t}$ yields the (nonquantum) $F$-polynomial $F_{j; t}^{B^{0}; t_{0}}$. 
\end{enumerate}
\end{theorem}
\begin{proof} We start by proving the existence of $F_{j; t} \in \mathcal{F}(R)$ which satisfy (\ref{fpoly-eqn-h}) so that $\lambda_{j; t} = 0$ if $\tilde{B}^{0}$ is principal.  Proceed by induction on the distance of $t$ from the initial vertex $t_{0}$ in the tree $\mathbb{T}_{n}$.  For $1 \leq j \leq n$, $X_{j; t_{0}} = M_{0}(\textbf{e}_{j})$ and $\tilde{\gg}_{j; t}^{\tilde{B}^{0}; t_{0}} = \ee_{j}$, so take $F_{j; t_{0}} = 1$ and $\lambda_{j; t_{0}} = 0$.  Now suppose that for some vertex $t \in \mathbb{T}_{n}$, each $F_{j; t}$ has been defined satisfying (\ref{fpoly-eqn-g}) and (\ref{fpoly-eqn-h}). Let $t' \in \mathbb{T}_{n}$ such that $t \frac{k}{\hspace{1cm}} t'$.  If $j \neq k$, then $X_{j; t'} = X_{j; t}$ and $\tilde{\gg}_{j; t'}^{\tilde{B}^{0}; t_{0}} = \tilde{\gg}_{j; t}^{\tilde{B}^{0}; t_{0}}$, so put $F_{j; t'} = F_{j; t}$. Now consider the cluster variable $X_{k; t'}$. Using Proposition \ref{prop:quantum-exchrel} and the fact that $[b]_{+} = b + [-b]_{+}$ for any $b \in \ZZ$, 
\begin{eqnarray}
X_{k; t'} = M_{t}(-\textbf{e}_{k} + \sum_{i = 1} ^{m} [-b_{ik} ^{t}]_{+} \textbf{e}_{i} + \tilde{\bb}^{k; t}) + M_{t}(-\textbf{e}_{k} + \sum_{i = 1} ^{m} [-b_{ik} ^{t}]_{+} \textbf{e}_{i})   \\
= q^{\frac{1}{2}d_{k}}M_{t}(\tilde{\bb}^{k; t})M_{t}(-\textbf{e}_{k} + \sum_{i = 1} ^{m} [-b_{ik} ^{t}]_{+} \textbf{e}_{i}) + M_{t}(-\textbf{e}_{k} + \sum_{i = 1} ^{m} [-b_{ik} ^{t}]_{+} \textbf{e}_{i})   \\
= (q^{\frac{1}{2}d_{k}} \hat{Y}_{k; t} + 1)M_{t}(-\textbf{e}_{k} + \sum_{i = 1} ^{m} [-b_{ik} ^{t}]_{+} \textbf{e}_{i}).   \hspace{4.5cm} \label{clustervareqn}
\end{eqnarray}
The second equality follows from the fact that for $j \in [1, n]$, $i \in [1, m]$
\begin{eqnarray} \label{quasicomm1}
\Lambda_{t}(\tilde{\bb}^{j; t}, \textbf{e}_{i}) = \delta_{ij}d_{j},
\end{eqnarray}
which is  a restatement of the compatibility condition between $\tilde{B}^{t}$ and $\Lambda_{t}$.  

Some notation that will be used throughout this section and the next: let $G_{1}, \ldots, G_{p}$ be elements of a skew-field.  Then define 
\begin{eqnarray}
\prod^{\rightarrow}_{i \in [1, p]} G_{i} = G_{1} \ldots G_{p}.
\end{eqnarray}
The elements $G_{1}, \ldots, G_{p}$ are not necessarily commutative, so this product notation establishes a fixed order in which the elements are to be multiplied.

Let $\rho_{1} \in \frac{1}{2} \mathbb{Z}$ be such that  $M_{t}(-\textbf{e}_{k} + \sum_{i = 1} ^{m} [-b_{ik} ^{t}]_{+} \textbf{e}_{i})$ is equal to 
\begin{eqnarray} \label{rho1}
 q^{\rho_{1}} M_{t}(-\textbf{e}_{k}) \left( \prod^{\rightarrow}_{i \in [1, n]} M_{t}([-b^{t}_{ik}]_{+} \textbf{e}_{i}) \right) M_{0}(\sum_{i = n + 1} ^{m} [-b_{ik} ^{t}]_{+} \textbf{e}_{i})
\end{eqnarray}
Thus, 
\begin{eqnarray}
X_{k; t'}  =  q^{\rho_{1}}(q^{\frac{1}{2}d_{k}}\hat{Y}_{k; t} + 1) X_{k; t}^{-1} \left(\prod^{\rightarrow}_{i \in [1, n]} X_{i; t}^{[-b^{t}_{ik}]_{+}}\right) M_{0}(\sum_{i = n + 1} ^{m} [-b_{ik} ^{t}]_{+} \textbf{e}_{i}).
\end{eqnarray}
Using the expressions given at (\ref{fpoly-eqn-h}) for cluster variables at the vertex $t$, 

\begin{eqnarray}  \label{Xlt1}
\noindent X_{k; t'} & = & q^{\rho_{1} + \lambda'}(q^{\frac{1}{2}d_{k}}\hat{Y}_{k; t} + 1) (F_{k; t}(\hat{Y}) M_{0}(\tilde{\gg}^{\tilde{B}^{0}; t_{0}}_{k; t}))^{-1}  \\
\nonumber & & \hspace{0.7cm} \times \left( \prod^{\rightarrow}_{i \in [1, n]} (F_{i; t}(\hat{Y}) M_{0}(\tilde{\gg}^{\tilde{B}^{0}; t_{0}}_{i; t}))^{[-b^{t}_{ik}]_{+}} \right) M_{0}(\sum_{i = n + 1} ^{m} [-b_{ik} ^{t}]_{+} \textbf{e}_{i}),
\end{eqnarray}
where $\lambda' \in \frac{1}{2}\mathbb{Z}$ is given by 
\begin{eqnarray} \label{lambda2}
\lambda' =  - \lambda_{k; t} + \sum_{i = 1} ^{n} [-b_{ik} ^{t}]_{+} \lambda_{i; t}.
\end{eqnarray}

Observe that $\lambda'$ does not depend on $\mathcal{A}$, only on the choice of $\tilde{B}^{0}$, and that $\lambda' = 0$ if $\tilde{B}^{0}$ is principal.

By (\ref{quasicomm1}), the elements $\hat{Y}_{j}$ and $X_{i}$ obey the following quasi-commutation relations for $j \in [1, n]$, $i \in [1, m]$:

\begin{eqnarray} \label{quasicomm2}
\hat{Y}_{j}X_{i} = q^{\delta_{ij}d_{j}}X_{i}\hat{Y}_{j}.
\end{eqnarray}
Observe that this quasi-commutation relation only depends on the entries of $D$.  Using (\ref{quasicomm2}), we can move all $M_{0}(\tilde{\gg}^{\tilde{B}^{0}; t_{0}}_{i; t})$ to the right in (\ref{Xlt1}), from which it follows that 

\begin{eqnarray}  \label{Xlt2}
X_{k; t'} & = & q^{\rho_{1} + \lambda'}(q^{\frac{1}{2}d_{k}}\hat{Y}_{k; t} + 1) P_{k; t'}(\hat{Y}) M_{0}(\tilde{\gg}^{\tilde{B}^{0}; t_{0}}_{k; t})^{-1} \\
\nonumber & & \hspace{1cm} \times  \left( \prod^{\rightarrow}_{i \in [1, n]} M_{0}(\tilde{\gg}^{\tilde{B}^{0}; t_{0}}_{i; t})^{[-b^{t}_{ik}]_{+}}  \right)  M_{0}(\sum_{i = n + 1} ^{m} [-b_{ik} ^{t}]_{+} \textbf{e}_{i}), 
\end{eqnarray}
where $P_{k; t'}$ is some element of $\mathcal{F}(R)$.  Observe that $P_{k; t'}$ does not depend on $\tilde{B}^{0}$ or the choice of quantum cluster algebra $\mathcal{A}$ for the following reasons: for each $j \in [1, n]$, $F_{j; t}$ does not depend on $\mathcal{A}$ or $\tilde{B}^{0}$, by the inductive hypothesis; the first $n$ coordinates of each $\tilde{\gg}^{\tilde{B}^{0}; t_{0}}_{j; t}$ are equal to $\gg_{j; t}$, which only depend on $B^{0}$ and $D$, and the remaining $m - n$ coordinates correspond to coefficient variables $X_{n + 1}, \ldots, X_{m}$, which commute with the $\hat{Y}_{j}$ elements; finally, the powers $[-b^{t}_{ik}]_{+}$ of the $M_{0}(\tilde{\gg}_{i; t}^{\tilde{B}^{0}; t_{0}})$ are in the principal part of $\tilde{B}^{t}$, which is the same for every quantum cluster algebra presently under consideration.

Write 

\begin{eqnarray}
\tilde{\gg}'_{k; t'} = -\tilde{\gg}^{\tilde{B}^{0}; t_{0}}_{k; t} +  \sum_{i =  1} ^{m} [-b_{ik} ^{t}]_{+} \tilde{\gg}^{\tilde{B}^{0}; t_{0}}_{i; t} 
\end{eqnarray}
\noindent and let $\rho_{2} \in \frac{1}{2} \mathbb{Z}$ such that   
\begin{eqnarray} \label{rho2}
M_{0}(\tilde{\gg}'_{k; t'}) =  q^{ \rho_{2}} M_{0}(-\tilde{\gg}_{k; t}) \left( \prod_{i \in [1, n]}^{\rightarrow} M_{0}(\tilde{\gg}^{\tilde{B}^{0}; t_{0}}_{i; t})^{[-b^{t}_{ik}]_{+}} \right) \\
\times M_{0}(\sum_{i = n + 1} ^{m} [-b_{ik} ^{t}]_{+} \textbf{e}_{i})
\end{eqnarray}
Put $\rho = \rho_{1} - \rho_{2}$.  Then 
\begin{eqnarray}  \label{Xlt3}
X_{k; t'} & = & q^{\rho + \lambda'}(q^{\frac{1}{2}d_{k}}\hat{Y}_{k; t} + 1) P_{k; t'}(\hat{Y}) M_{0}(\tilde{\gg}'_{k; t'}).
\end{eqnarray}
To bring this expression closer to (\ref{fpoly-eqn-h}), we must show that $\hat{Y}_{k; t}$ can be expressed as a subtraction-free rational expression in $\hat{Y}_{1}, \ldots, \hat{Y}_{n}$.   

\begin{lemma} \label{yhatprop} Let $t \frac{k}{\hspace{1cm}} t'$ be vertices of $\mathbb{T}_{n}$, let $j \in [1, n]$, and write $b = b^{t}_{kj} = -b^{t'}_{kj}$.  Then
\begin{eqnarray}  \label{yhatrecc}
\hat{Y}_{j; t'} = \left\{   \begin{array}{cl}
                            \hat{Y}_{j; t} \prod_{p = 0}^{|b| - 1} (1 + q^{(-d_{k}p - \frac{d_{k}}{2})}\hat{Y}_{k; t}) & \mbox{ if } b \leq 0, j \neq k  \\
                   \hat{Y}_{j; t}\hat{Y}_{k; t}^{b} \prod_{p = 0}^{b - 1} (\hat{Y}_{k; t} + q^{(-d_{k}p - \frac{d_{k}}{2})})^{-1} & \mbox{ if } b \geq 0, j \neq k \\
                        \hat{Y}_{k; t}^{-1} & \mbox{ if } j = k
                            \end{array}
                        \right.
\end{eqnarray}
\end{lemma} 
\begin{proof} 

\

\noindent \textit{Case 1: $b \leq 0$, $j \neq k$.}  

\noindent Using (\ref{eq:toric-frame-mutation}), 
\begin{eqnarray} 
\hat{Y}_{j; t'} = \sum_{p = 0}^{|b|} \left(\!\! \begin{array}{c} |b| \\ p  \end{array}  \!\!\right)_{q^{d_{k}/2}}   M_{t}(E_{-}\tilde{\bb}^{j; t'} - p\tilde{\bb}^{k; t}), 
\end{eqnarray}
where $E_{-} = (e_{il})$ is the $m \times m$ matrix whose entries are given at (\ref{eq:E-epsilon}).
For each $p = 1, \ldots, |b|$, 
\begin{eqnarray} \label{Eminus}
E_{-}\tilde{\bb}^{j; t'} & = & -b^{t'}_{kj}\textbf{e}_{k} + \sum_{i \in [1, m], i \neq k} (b^{t'}_{ij} + b^{t'}_{kj}[b^{t}_{ik}]_{+})\textbf{e}_{i}  \\
\nonumber & = &  b^{t}_{kj}\textbf{e}_{k} + \sum_{i \in [1, m], i \neq k} (b_{ij}^{t} + [bb^{t}_{ik}]_{+}sgn(b) - b[b^{t}_{ik}]_{+})\textbf{e}_{i}   
\end{eqnarray}
Since $b \leq 0$, one can check that 
\begin{eqnarray}
E_{-}\tilde{\bb}^{j; t'} = b^{t}_{kj}\textbf{e}_{k} + \sum_{i \in [1, m], i \neq k} (b_{ij}^{t} - bb^{t}_{ik})\textbf{e}_{i} = \tilde{\bb}^{j; t} - b\tilde{\bb}^{k; t}.
\end{eqnarray}
It follows that 
\begin{eqnarray}
\hat{Y}_{j; t'} & = & \sum_{p = 0}^{|b|} \left(\!\! \begin{array}{c} |b| \\ p  \end{array}  \!\!\right)_{q^{d_{k}/2}}   M_{t}(\tilde{\bb}^{j; t} - b\tilde{\bb}^{k; t} - p\tilde{\bb}^{k; t}) \\
\nonumber & = &  \sum_{p = 0}^{|b|} \left(\!\! \begin{array}{c} |b| \\ p  \end{array}  \!\!\right)_{q^{d_{k}/2}}   M_{t}(\tilde{\bb}^{j; t} + p\tilde{\bb}^{k; t})  \\
\nonumber & = & \sum_{p = 0}^{|b|} \left(\!\! \begin{array}{c} |b| \\ p  \end{array}  \!\!\right)_{q^{d_{k}/2}}   q^{-d_{k}|b|p/2}  M_{t}(\tilde{\bb}^{j; t})M_{t}(\tilde{\bb}^{k; t})^{p}  \\
\nonumber & = &  M_{t}(\tilde{\bb}^{j; t}) \prod_{p = 0}^{|b| - 1} [1 + (q^{d_{k}/2})^{(|b| - 1 - 2p)}(q^{-d_{k}|b|/2}M_{t}(\tilde{\bb}^{k; t}))]
\end{eqnarray}
The last equality follows from the $t$-binomial formula.  The proposition in this case follows after simplifying this last expression.

\

\noindent \textit{Case 2: $b \geq 0$, $j \neq k$.}

\noindent In this case, $\hat{Y}_{j; t'} = M_{t'}(-\tilde{\bb}^{j; t'})^{-1}$, and 

\begin{eqnarray} 
M_{t'}(-\tilde{\bb}^{j; t'}) = \sum_{p = 0}^{|b|} \left(\!\! \begin{array}{c} |b| \\ p  \end{array}  \!\!\right)_{q^{d_{k}/2}}   M_{t}(-E_{-}\tilde{\bb}^{j; t'} - p\tilde{\bb}^{k; t})
\end{eqnarray}
Using the expression at (\ref{Eminus}), 
\begin{eqnarray}
E_{-}\tilde{\bb}^{j; t'} = b^{t}_{kj}\textbf{e}_{k} + \sum_{i \in [1, m], i \neq k} b_{ij}^{t}\textbf{e}_{i} = \tilde{\bb}^{j; t}.
\end{eqnarray}
Thus, 
\begin{eqnarray}
\hat{Y}_{j; t'} & = & \left[\sum_{p = 0}^{|b|} \left(\!\! \begin{array}{c} |b| \\ p  \end{array}  \!\!\right)_{q^{d_{k}/2}} M_{t}(-\tilde{\bb}^{j; t} - p\tilde{\bb}^{k; t})\right]^{-1} \\
\nonumber & = & \left[\sum_{p = 0}^{|b|} \left(\!\! \begin{array}{c} |b| \\ p  \end{array}  \!\!\right)_{q^{d_{k}/2}} q^{-d_{k}|b|p/2}M_{t}(\tilde{\bb}^{k; t})^{-p}M_{t}(\tilde{\bb}^{j; t})^{-1}\right]^{-1} \\
\nonumber & = & M_{t}(\tilde{\bb}^{j; t}) \prod_{p = 0}^{|b| - 1} \left[(1 + (q^{d_{k}/2})^{|b| - 1 - 2p})(q^{-d_{k}|b|/2}M_{t}(\tilde{\bb}^{k; t})^{-1} )\right]^{-1}
\end{eqnarray}
where the last expression follows from the $t$-binomial formula (\ref{eqn:t-binomial-formula}). Additional simplification yields the proposition in this case.

\

\noindent \textit{Case 3: $j = k$} 

\noindent In this case, $b = 0$.  Using (\ref{eq:toric-frame-mutation}) again, 
\begin{eqnarray}
\hat{Y}_{k; t'} = M_{t}(E_{-}\tilde{\bb}^{k; t'}) = M_{t}(\tilde{\bb}^{k; t'}) = M_{t}(-\tilde{\bb}^{k; t}) = M_{t}(\tilde{\bb}^{k; t})^{-1}.
\end{eqnarray}
\end{proof}
  
 \begin{remark} The recurrence relations for the $\hat{Y}_{i; t}$ which appear in Lemma \ref{yhatprop} are essentially the same as certain quantum mutation maps which occur in quantum spaces $\mathcal{X}_{q}$ as defined in \cite{fg} (see Lemma 3.2 in \emph{loc. cit.}).
  \end{remark} 
  
By Lemma \ref{yhatprop}, one may show that there exists a unique $\mathcal{Y}_{k; t} \in \mathcal{F}(R)$ such that $\mathcal{Y}_{k; t}(\hat{Y}) = \hat{Y}_{k; t}$; furthermore, $\mathcal{Y}_{k; t}$ can be expressed as a subtraction-free rational expression in $Z_{1}, \ldots, Z_{n}$ depending only on $B^{0}$ and $D$, and not on $\tilde{B}^{0}$ or $\mathcal{A}$.    Set
\begin{eqnarray}
G_{k; t} = q^{\frac{1}{2}d_{k}}\mathcal{Y}_{k; t} + 1 \in \mathcal{F}(R).
\end{eqnarray}
By (\ref{hvec-rec2}),
\begin{eqnarray}
\tilde{\gg}_{k; t'}^{\tilde{B}^{0}; t_{0}} = \tilde{\gg}_{k; t}' - \sum_{i = 1} ^{n} [-b^{\bullet, t}_{ik}]_{+}\tilde{\bb}^{i}.
\end{eqnarray}

Let $\lambda''$ be an element of $\frac{1}{2}\mathbb{Z}$ satisfying 

\begin{eqnarray}
M_{0}(\tilde{\gg}_{k; t'}) & = & q^{\lambda''} M_{0}(-[-b^{\bullet, t}_{nk}]_{+}\tilde{\bb}^{n; t_{0}}) \ldots M_{0}(-[-b^{\bullet, t}_{1k}]_{+} \tilde{\bb}^{1; t_{0}}) M_{0}(\tilde{\gg}'_{k; t'})\\
\nonumber & = & q^{\lambda''}  \hat{Y}_{n}^{-[-b^{\bullet, t}_{nk}]_{+}} \ldots \hat{Y}_{1} ^{-[-b^{\bullet, t}_{1k}]_{+}}M_{0}(\tilde{\gg}'_{k; t'})
\end{eqnarray}
Using (\ref{yhat-comm2}) and Proposition \ref{h-vec-props}, $\lambda''$ can be written explicitly as 

\begin{eqnarray}
\frac{1}{2} \Lambda_{0}(\sum_{i = 1} ^{n} [-b^{\bullet, t}_{ik}]_{+}\tilde{\bb}^{i}, \tilde{\gg}'_{k; t'}) + \frac{1}{2} \sum_{1 \leq i < j \leq n} \Lambda_{0}([-b^{\bullet, t}_{ik}]_{+}\tilde{\bb}^{i}, [-b^{\bullet, t}_{jk}]_{+}\tilde{\bb}^{j})  \\
\nonumber = \frac{1}{2} \Lambda_{0}(\sum_{i = 1} ^{n} [-b^{\bullet, t}_{ik}]_{+}\tilde{\bb}^{i}, \gg_{k; t'}) + \frac{1}{2} \sum_{1 \leq i < j \leq n} [-b^{\bullet, t}_{ik}]_{+}[-b^{\bullet, t}_{jk}]_{+}d_{i}b_{ij}  \label{lambda3}
 \end{eqnarray}
Thus, $\lambda''$ will only depend on $B^{0}$, $D$, not on $\tilde{B}^{0}$ or on $\mathcal{A}$.   Let $\rho_{\bullet}$ be the value of $\rho$ as obtained above when $\tilde{B}^{0}$ is the principal matrix corresponding to $B^{0}$.  Let   

\begin{eqnarray}  \label{fpolydef0}
F_{k; t'} = q^{\rho_{\bullet} - \lambda''} G_{k; t}P_{k; t'}Z_{1}^{[-b^{\bullet, t}_{1k}]_{+}} \ldots Z_{n}^{[-b^{\bullet, t}_{nk}]_{+}},
\end{eqnarray}
\begin{eqnarray} \label{eq:lambda-rec}
\lambda_{k; t'} = \rho - \rho_{\bullet} + \lambda'.
\end{eqnarray}

By (\ref{Xlt3}),  

\begin{eqnarray}
X_{k; t'} & = & q^{\rho + \lambda'}(q^{\frac{1}{2}d_{k}}\hat{Y}_{k; t} + 1) P_{k; t'}(\hat{Y}) M_{0}(\tilde{\gg}'_{k; t'}) \\
& = & q^{\rho - \rho_{\bullet} + \lambda' + \lambda''} F_{k; t'}(\hat{Y}) \hat{Y}_{n}^{-[-b^{\bullet, t}_{nk}]_{+}} \ldots \hat{Y}_{1} ^{-[-b^{\bullet, t}_{1k}]_{+}}M_{0}(\tilde{\gg}'_{k; t'}) \\
& = & q^{\rho - \rho_{\bullet} + \lambda'} F_{k; t'}(\hat{Y}) M_{0}(\tilde{\gg}_{k; t'})    \label{fpolydef0b}
\end{eqnarray}

The proof of the existence of $F_{j; t} \in \mathcal{F}(R)$ satisfying (\ref{fpoly-eqn-g}) and (\ref{fpoly-eqn-h}) now follows by induction.

Note that $\lambda_{k; t'} = 0$ if $\tilde{B}^{0}$ is the principal matrix corresponding to $B^{0}$.   To check that $F_{k; t'}$ is independent of the choice of $\tilde{B}^{0}$ or $\mathcal{A}$, and that $\lambda_{k; t'}$ depends on $\tilde{B}^{0}$ but not $\mathcal{A}$, we must prove that $\rho$ depends on $\tilde{B}^{0}$, but not on the choice of quantum cluster algebra $\mathcal{A}$.  For $t \in \mathbb{T}_{n}$, $i, j \in [1, m]$, define
\begin{eqnarray}
\rho^{t}_{ij}(\tilde{B}^{0}) = \Lambda_{t}(e_{i}, e_{j}) - \Lambda_{0}(\tilde{\gg}_{i; t}^{\tilde{B}^{0}; t_{0}}, \tilde{\gg}_{j; t}^{\tilde{B}^{0}; t_{0}}).
\end{eqnarray}
With this notation and the definitions of $\rho_{1}, \rho_{2}$ given at (\ref{rho1}),  (\ref{rho2}),  it follows that 
\begin{eqnarray}
\rho & = & \rho_{1} - \rho_{2} \\
 \nonumber    & = & -\frac{1}{2} \sum_{i = 1}^{m} [-b^{t}_{ik}]_{+} \rho^{t}_{ik}(\tilde{B}^{0}) 
           + \frac{1}{2} \sum [-b^{t}_{ik}]_{+} [-b^{t}_{jk}]_{+} \rho_{ij}^{t}(\tilde{B}^{0}),
\end{eqnarray}
where the last summation ranges over $j \in [1, n], i \in [1, m]$ such that $j < i$.  Thus, the assertion about $\rho$ follows from the next lemma. 

\begin{lemma} \label{lemma:rho-props} Let $i, j \in [1, m]$, and let $\rho^{t}_{ij} = \rho^{t}_{ij}(\tilde{B}^{0})$. 
The following properties hold:
\begin{eqnarray} 
\label{rho-init}  \rho^{t_{0}}_{ij} & = & 0 \\
\label{rho-prop1} \rho^{t}_{ij} & = & -\rho^{t}_{ji} \\
\label{rho-prop2} \rho^{t}_{ii} & = & 0 \\
\label{rho-prop3} \rho^{t}_{ij} & = & 0 \hspace{0.5cm} \mbox{ if } i \in [n + 1, m] \mbox{ or } j \in [n + 1, m].
\end{eqnarray}
If $t \frac{k}{\hspace{1cm}} t'$ in $\mathbb{T}_{n}$, then 
\begin{eqnarray}
\label{rho-rec1} \rho^{t'}_{ij} & = & \rho^{t}_{ij} \hspace{0.5cm} \mbox{ if } i, j \neq k \\
\label{rho-rec2} \rho^{t'}_{ik} & = & -\rho^{t}_{ik} + \sum_{\ell = 1}^{m} [-b_{ik}^{t}]_{+} \rho^{t}_{i, \ell} - \sum_{\ell = 1}^{n} [-b^{\bullet, t}_{\ell + n, k}]_{+} g_{\ell} d_{\ell} \hspace{0.5cm} \mbox{ if } i \neq k, i \in [1, n] 
\end{eqnarray}
where $\gg_{i; t}^{B^{0}; t_{0}} = (g_{1}, \ldots, g_{n})$.

Consequently, $\rho^{t}_{ij}(\tilde{B}^{0})$ depends only on $\tilde{B}^{0}$, not on the choice of quantum cluster algebra $\mathcal{A}$.
\end{lemma}
\begin{proof} Equations (\ref{rho-init}), (\ref{rho-rec1}) follow from (\ref{hvec-init}), (\ref{hvec-rec1}), respectively, while (\ref{rho-prop1}), (\ref{rho-prop2}) follow from the fact that $\Lambda_{0}$, $\Lambda_{t}$ are skew-symmetric.  Also, (\ref{rho-prop3}) is clear when $i, j \in [n + 1, m]$.  Let $i \in [1, m]$, $i \neq k$.  From (\ref{eq:Lambda-mutation}),
\begin{eqnarray}
\Lambda_{t'}(e_{i}, e_{k})  =  -\Lambda_{t}(e_{i}, e_{k}) + \sum_{\ell = 1}^{m} [-b^{t}_{\ell k}]_{+} \Lambda_{t}(e_{i}, e_{\ell}) 
\end{eqnarray}
Using (\ref{hvec-rec2}), $\Lambda_{0}(\tilde{\gg}_{i; t'}^{\tilde{B}^{0}; t_{0}}, \tilde{\gg}_{k; t'}^{\tilde{B}^{0}; t_{0}})$ equals
\begin{eqnarray}
& \displaystyle \Lambda_{0}(\tilde{\gg}_{i; t}^{\tilde{B}^{0}; t_{0}}, -\tilde{\gg}_{k; t}^{\tilde{B}^{0}; t_{0}} + \sum_{\ell = 1}^{m} [-b^{t}_{\ell k}]_{+} \tilde{\gg}_{\ell; t}^{\tilde{B}^{0}; t_{0}} - \sum_{\ell = 1}^{n} [-b^{\bullet, t}_{\ell + n, k}]_{+} \tilde{b}^{\ell}) \\
\nonumber & \displaystyle \hspace{3cm} = -\Lambda_{0}(\tilde{\gg}_{i; t}^{\tilde{B}^{0}; t_{0}}, \tilde{\gg}_{k; t'}^{\tilde{B}^{0}; t_{0}}) + \sum_{\ell = 1}^{m} [-b^{t}_{\ell k}]_{+} \Lambda_{0}(\tilde{\gg}_{i; t}^{\tilde{B}^{0}; t_{0}}, \tilde{\gg}_{\ell; t}^{\tilde{B}^{0}; t_{0}}) \\
\nonumber & \displaystyle \hspace{4cm} - \sum_{\ell = 1}^{n} [-b^{\bullet, t}_{\ell + n, k}]_{+} \Lambda_{0}.(\tilde{\gg}_{i; t}^{\tilde{B}^{0}; t_{0}}, \tilde{b}^{\ell})
\end{eqnarray}
First, consider the case where $i \in [1, n]$. Then 
\begin{eqnarray}
\rho^{t'}_{ik} = -\rho^{t}_{ik} + \displaystyle \sum_{\ell = 1}^{m} [-b^{t}_{\ell k}]_{+}\rho^{t}_{i\ell} + \sum_{\ell = 1}^{n} [-b^{\bullet, t}_{\ell + n, k}]_{+} \Lambda_{0}(\tilde{\gg}_{i; t}^{\tilde{B}^{0}; t_{0}}, \tilde{b}^{\ell}).
\end{eqnarray}
By Proposition \ref{h-vec-props}, for $i \in [1, n]$, the first $n$ coordinates of $\tilde{\gg}_{i; t}^{\tilde{B}^{0}; t_{0}}$ are given by $\gg_{i; t}^{B^{0}; t_{0}}$.  By (\ref{quasicomm1}), 
\begin{eqnarray}
\sum_{\ell = 1}^{n} [-b^{\bullet, t}_{\ell + n, k}]_{+} \Lambda_{0}(\tilde{\gg}_{i; t}^{\tilde{B}^{0}; t_{0}}, \tilde{b}^{\ell}) = -\sum_{\ell = 1}^{n} [-b^{\bullet, t}_{\ell + n, k}]_{+} g_{\ell} d_{\ell}.
\end{eqnarray}
This proves (\ref{rho-rec2}).

Now suppose that $i \in [n + 1, m]$.  Then 
\begin{eqnarray} \label{rho-rec3}
\rho^{t'}_{ik} = -\rho^{t}_{ik} + \displaystyle \sum_{\ell = 1}^{m} [-b^{t}_{\ell k}]_{+}\rho^{t}_{i\ell}.
\end{eqnarray}
To prove that $\rho^{t}_{ij} = 0$ when $i$ or $j$ is in $[n + 1, m]$, use induction on the distance of $t$ from $t_{0}$, with (\ref{rho-init}) as the base of the induction, and (\ref{rho-prop1}), (\ref{rho-prop2}), (\ref{rho-prop3}) (known when both $i, j \in [n + 1, m]$), (\ref{rho-rec1}), and (\ref{rho-rec3}) as recurrence relations.
\end{proof}

Setting $q = 1$ in (\ref{fpoly-eqn-g}) yields
\begin{eqnarray}
x_{j; t} = F_{j; t}|_{q = 1}(\hat{y}_{1}, \ldots, \hat{y}_{n})x_{1}^{g_{1}}\ldots x_{n}^{g_{n}} \in \prinA(B^{0}, t_{0}),
\end{eqnarray}
where $\gg_{j; t} = (g_{1}, \ldots, g_{n})$.  Comparing this (\ref{eq:cluster-var-fpoly}), it follows that 
\begin{eqnarray}
F_{j; t}|_{q = 1}(\hat{y}_{1}, \ldots, \hat{y}_{n}) = F_{j; t}^{B^{0}; t_{0}}(\hat{y}_{1}, \ldots, \hat{y_{n}}).  
\end{eqnarray}
Since $\hat{y}_{1}, \ldots, \hat{y}_{n}$ are algebraically independent, part (3) of Theorem \ref{thm:quantum-fpoly}  follows.  

For the uniqueness of $F_{j; t}$, observe that if another $F'_{j; t} \in R$ satisfies equation (\ref{fpoly-eqn-g}), then $F_{j; t}(\hat{Y}) = F'_{j; t}(\hat{Y})$.  Since $\hat{Y}_{1}, \ldots, \hat{Y}_{n}$ are algebraically independent, it follows that $F_{j; t} = F'_{j; t}$.

We conclude the proof of part (1) of Theorem \ref{thm:quantum-fpoly} by showing that $F_{j; t}$ is a polynomial in $Z_{1}, \ldots, Z_{n}$ with coefficients in $\mathbb{Z}[q^{\pm \frac{1}{2}}]$.  The quantum Laurent phenomenon (Theorem \ref{thm:quantum-laurent}) implies that each cluster variable $X_{j; t}$ in $\mathcal{A}_{\bullet}(B^{0}, D, \Lambda, t_{0})$ can be expressed as a Laurent polynomial in $X_{1}, \ldots, X_{2n}$ with coefficients in $\basering$, which implies the same is true for $F_{j; t}(\hat{Y})$.  Let $F_{j; t}(\hat{Y}) = P(X_{1}, \ldots, X_{2n})$ for some Laurent polynomial $P$ with coefficients in $\basering$, and write $F_{j; t} = A^{-1}C$ for some elements $A, C \in R$.  Then 
\begin{eqnarray} \label{cap-eqn}
C(\hat{Y}) & = & A(\hat{Y})P(X_{1}, \ldots, X_{2n}). 
\end{eqnarray}
If $F$ is a Laurent polynomial in $X_{1}, \ldots, X_{2n}$, then let the Newton polytope $\mbox{Newt}(F)$ of $F$ be the convex hull of the set 
\begin{eqnarray}
\{ (a_{1}, \ldots, a_{2n}) \in \ZZ^{2n} : X_{1}^{a_{1}} \ldots X_{2n}^{a_{2n}} \mbox{ has nonzero coefficient in } F \}.
\end{eqnarray}
Taking the Newton polytope of both sides of (\ref{cap-eqn}),
\begin{eqnarray}
\hspace{1cm} \mbox{Newt}(C(\hat{Y})) & = & \mbox{Newt}(A(\hat{Y})) + \mbox{Newt}(P(X_{1}, \ldots, X_{2n})) 
\end{eqnarray}
\noindent This implies that $\mbox{Newt}(P(X_{1}, \ldots, X_{2n}))$ is contained in the $\mathbb{Q}$-linear span of $\tilde{\bb}^{1}, \ldots, \tilde{\bb}^{n}$.  Thus, each exponent vector of $P(X_{1}, \ldots, X_{2n})$ can be expressed as a $\mathbb{Q}$-linear combination of $\tilde{\bb}^{1}, \ldots, \tilde{\bb}^{n}$.  However, the last $n$ coordinates of the vector $\tilde{\bb}^{i}$ is given by the vector $\ee_{i} \in \ZZ^{n}$, so each exponent vector must in fact be a $\mathbb{Z}$-linear combination of $\tilde{\bb}^{1}, \ldots, \tilde{\bb}^{n}$.  This proves $P(X_{1}, \ldots, X_{2n}) = F_{j; t}(\hat{Y})$ is a Laurent polynomial in $\Yhats$.  Since $\Yhats$ are algebraically independent, this implies $F_{j; t}$ is a Laurent polynomial in $Z_{1}, \ldots, Z_{n}$ with coefficients in $\ZZ[q^{\pm \frac{1}{2}}]$.

Now consider the Newton polytope $N(F_{j; t})$ of $F_{j; t}$ with respect to $Z_{1}, \ldots, Z_{n}$, i.e., $N(F_{j; t})$ is the convex hull in $\mathbb{R}^{n}$ of the set 
\begin{eqnarray}
\{ \aa \in \ZZ^{n} : Z^{\aa} \mbox{ has nonzero coefficient in } F_{j; t} \}.
\end{eqnarray}
Denote by $N(F)$ the Newton polytope of the polynomial $F(u_{1}, \ldots, u_{n})$ with respect to $u_{1}, \ldots, u_{n}$, which is defined in a analogous way $N(F_{j; t})$.   By following the proof of the existence of $F_{j; t}$, it is easy to see that $F_{j; t}$ is obtained from $Z_{1}, \ldots, Z_{n}$ via a series of subtraction-free rational transformations.  This means that for each vertex $\textbf{c} \in N(F_{j; t})$, the monomial in  $F_{j; t}$ with exponent vector $\textbf{c}$ will have a coefficient that can be expressed as a subtraction-free rational expression in $\ZZ[q^{\pm \frac{1}{2}}]$.  Thus, setting $q = 1$ will not shrink the Newton polytope.  Consequently,
\begin{eqnarray} \label{fpoly-newton-polytope}
N(F_{j; t}) = N(F_{j; t}^{B^{0}; t_{0}}),
\end{eqnarray}
where $F_{j; t}^{B^{0}; t_{0}}$ is the nonquantum $F$-polynomial.  Since $F_{j; t}$ is a polynomial in $u_{1}, \ldots, u_{n}$, the polytope $N(F_{j; t}^{B^{0}; t_{0}})$ does not contain any points with negative coordinates. The same is true of $N(F_{j; t})$, forcing $F_{j; t}$ to be a polynomial in $Z_{1}, \ldots, Z_{n}$.
\end{proof}

\begin{example} \label{example:simple-quantum-fpolys} For all $j \in [1, n]$, $F_{j; t_{0}} = 1$.  If $t_{0} \frac{k}{\hspace{1cm}} t$ in $\mathbb{T}_{n}$, then the $\gg$-vector recurrences (Proposition \ref{g-vec-rec-b}) imply that 
\begin{eqnarray}
\gg_{k; t} = -\ee_{k} + \sum_{i = 1}^{n} [-b_{ik}^{0}]_{+}\ee_{i}.
\end{eqnarray}
From equation (\ref{clustervareqn}), it follows that 
\begin{eqnarray}  \label{eqn:fpoly-simple}
F_{k; t} = q^{d_{k}/2}Z_{k} + 1.
\end{eqnarray}
For $j \neq k$, $F_{j; t} = 1$.
\end{example}

\section{Properties of Quantum $F$-polynomials} \label{section:quantum-f-polys-props}

We continue to use the same notation as in the previous section.     Part (2) of Theorem \ref{thm:quantum-fpoly} may be strengthened under a certain condition which is conjectured to be true in general.

\begin{theorem} \label{thm:no-q-factor} If the nonquantum $F$-polynomial $F_{j; t}^{B^{0}; t_{0}}$ has nonzero constant term for all $j \in [1, n]$, $t \in \mathbb{T}_{n}$, then for any cluster variable $X_{j; t} \in \mathcal{A}$, 
\begin{eqnarray} \label{fpoly-eqn-h2}
X_{j; t} = F_{j; t}(\hat{Y})M_{0}(\tilde{\gg}_{j; t}^{\tilde{B}^{0}; t_{0}}).
\end{eqnarray}
\end{theorem}

\begin{proof} We need to show that $\lambda_{j; t} = 0$ (see (\ref{fpoly-eqn-h})).  Proceed by induction on the distance of the vertex $t$ from $t_{0}$ in $\mathbb{T}_{n}$, with $\lambda_{j; t_{0}} = 0$ already established in the proof of Theorem \ref{thm:quantum-fpoly}.  Let $t \frac{k}{\hspace{1cm}} t'$ in $\mathbb{T}_{n}$, and suppose that $\lambda_{j; t} = 0$ for all $j \in [1, n]$.  From the proof of Theorem \ref{thm:quantum-fpoly}, $\lambda_{j; t'} = \lambda_{j; t} = 0$ for $j \neq k$.  Note that $\lambda_{k; t'}$ is given by equations (\ref{lambda2}) and (\ref{eq:lambda-rec}).  To show that $\lambda_{k; t'} = 0$, it suffices to prove that $\rho = \rho_{\bullet} = 0$.  The proof follows by induction after the next lemma is proven.
\end{proof}

\noindent \begin{lemma} \label{lemma:rho-zero} Suppose that for all $j \in [1, n], t \in \mathbb{T}_{n}$, the nonquantum $F$-polynomial $F_{j; t} ^{B^{0}, t_{0}}$ has nonzero constant term.  Then $\Lambda_{t}(\ee_{i}, \ee_{j}) = \Lambda_{0}(\gg_{i; t}, \gg_{j; t})$,  and consequently, $\rho^{t}_{ij}(\tilde{B}^{0}) = 0$ for all $i, j \in [1, m]$.   
\end{lemma}
\begin{proof} First, assume that $\tilde{B}^{0}$ is principal.  Let 
\begin{eqnarray}
\lambda & = & \Lambda_{t}(\ee_{i}, \ee_{j}) \\
\lambda' & = & \Lambda_{0}(\gg_{i; t}, \gg_{j; t})
\end{eqnarray}
Use the convention that $F_{\ell; t} = 1$ if $\ell \in [n + 1, m]$.  
The expressions for $X_{i; t}$, $X_{j; t}$ from (\ref{fpoly-eqn-g}) imply that 
\begin{eqnarray}
F_{i; t}(\hat{Y})M_{0}(\gg_{i; t}) F_{j; t}(\hat{Y})M_{0}(\gg_{j; t}) = q^{\lambda}F_{j; t}(\hat{Y})M_{0}(\gg_{j; t}) F_{i; t}(\hat{Y})M_{0}(\gg_{i; t}). 
\end{eqnarray}
For some $P_{i}, P_{j} \in \mathcal{F}(R)$, 
\begin{eqnarray} \label{eqnlem1}
F_{i; t}(\hat{Y}) P_{j}(\hat{Y}) M_{0}(\gg_{i; t}) M_{0}(\gg_{j; t}) = q^{\lambda} F_{j; t}(\hat{Y}) P_{i}(\hat{Y}) M_{0}(\gg_{j; t}) M_{0}(\gg_{i; t}).
\end{eqnarray}
Since $M_{0}(\gg_{i; t}) M_{0}(\gg_{j; t}) = q^{\lambda'} M_{0}(\gg_{j; t}) M_{0}(\gg_{i; t}),$ it follows that 
\begin{eqnarray} \label{eqnlem2}
q^{\lambda'} F_{i; t}(\hat{Y}) P_{j}(\hat{Y}) = q^{\lambda} F_{j; t}(\hat{Y}) P_{i}(\hat{Y}).
\end{eqnarray}
Since $\tilde{B}^{0}$ has full rank, $\hat{Y}_{1}, \ldots, \hat{Y}_{n}$ are algebraically independent.  Thus, 
\begin{eqnarray} \label{eqnlem3}
q^{\lambda'}F_{i; t}P_{j} = q^{\lambda}F_{j; t}P_{i}.
\end{eqnarray}
By (\ref{fpoly-newton-polytope}), the fact that $F_{i; t} ^{B^{0}, t_{0}}$ has nonzero constant term implies that same thing about $F_{i; t} ^{B^{0}, D, \Lambda_{0}, t_{0}}$.  Observe that $F_{j; t}$ (resp. $F_{i; t}$) has the same constant term as $P_{j}$ (resp. $P_{i}$).  Considering the constant term of both sides of (\ref{eqnlem3}),  we conclude that $\lambda = \lambda'$.  As a consequence, (\ref{rho-rec2}) implies that for any $t \in \mathbb{T}_{n}$, $i, k \in [1, n]$, $i \neq k$, 
\begin{eqnarray} \label{eqnlem4}
\sum_{\ell = 1}^{n} [-b^{\bullet, t}_{\ell + n, k}]_{+} g_{\ell} d_{\ell} = 0,
\end{eqnarray}
where $\gg_{i; t}^{B^{0}; t_{0}} = (g_{1}, \ldots, g_{n})$.

To prove that $\rho^{t}_{ij}(\tilde{B}^{0}) = 0$ in the general case where $\tilde{B}^{0}$ is any exchange matrix with principal part $B^{0}$, apply induction on the distance of $t$ from $t_{0}$ in $\mathbb{T}_{n}$, and use Lemma \ref{lemma:rho-props} and (\ref{eqnlem4}).
\end{proof}

\begin{remark} It was conjectured that nonquantum $F$-polynomials always have constant term 1 in \cite[Conjecture 5.4]{coefficients}.  This conjecture was proven in \cite{qps2} for the case when $B^{0}$ is skew-symmetric, and also in \cite{fukeller} in the situation where the cluster algebra admits a certain categorification.
\end{remark}

Clearly, one can prove the existence and uniqueness of analogous "right" quantum $F$-polynomials by similar arguments as given in the proof of Theorem \ref{thm:quantum-fpoly} .  However, the next proposition demonstrates that these polynomials can easily be computed once the "left" $F$-polynomials are known.  Following \cite[Section 6]{quantum}, define the $\ZZ$-linear \textit{bar-involution} $X \mapsto \overline{X}$ on a quantum cluster algebra $\mathcal{A}$ with initial quantum seed $(M_{0}, \tilde{B}^{0})$ by setting
\begin{eqnarray}
\overline{q^{r/2}M_{0}(\textbf{c})} = q^{-r/2}M_{0}(\textbf{c}) \hspace{0.5cm} (r \in \ZZ, \textbf{c} \in \ZZ^{m}) 
\end{eqnarray}
We may also define a $\ZZ$-linear \textit{bar-involution} $X \mapsto \overline{X}$ on $R_{B^{0}, D}$ by
\begin{eqnarray}
\overline{q^{r/2}Z^{\textbf{c}}} = q^{-r/2}Z^{\textbf{c}} \hspace{0.5cm} (r \in \ZZ, \textbf{c} \in \ZZ^{n})
\end{eqnarray}
For $X, Y \in \mathcal{A}$ (resp. $X, Y \in R_{B^{0}, D}$), it's easy to verify that $\overline{XY} = \overline{Y}\hspace{0.1cm}\overline{X}$.

\begin{proposition} \label{prop:fpoly-right} Let $j \in [1, n]$, $t \in \mathbb{T}_{n}$.  Then $\overline{F}_{j; t}$ is the unique polynomial in $Z_{1}, \ldots, Z_{n}$ with coefficients in $\ZZ[q^{\pm \frac{1}{2}}]$ such that $X_{j; t} \in \prinA(B^{0}, D, \Lambda, t_{0})$ is given by 
\begin{eqnarray} \label{eq:fpoly-right}
X_{j; t} = M_{0}(\gg_{j; t})\overline{F_{j; t}}(\hat{Y}).
\end{eqnarray}
Furthermore, for any quantum cluster algebra $\mathcal{A}$ whose initial exchange matrix is $\tilde{B}^{0}$, the cluster variable $X_{j; t} \in \mathcal{A}$ is given by
\begin{eqnarray} \label{eqn1}
X_{j; t} = q^{-\lambda^{\tilde{B}^{0}, t_{0}}_{j; t}} M_{0}(\tilde{\gg}^{\tilde{B}^{0}, t_{0}}_{j; t})\overline{F}_{j; t}(\hat{Y}).  
\end{eqnarray}
Setting $q = 1$ and $Z_{i} = u_{i}$ in $\overline{F}_{j; t}$ yields the nonquantum $F$-polynomial $F^{B^{0}; t_{0}}_{j; t}$.
\end{proposition} 
\begin{proof}  By \cite[Proposition 6.2]{quantum}, if $(M, \tilde{B})$ is some quantum seed associated with $\mathcal{A}$, then $M(\textbf{c})$ is invariant under the bar-involution for any $\textbf{c} \in \ZZ^{2n}$.   In particular, cluster variables are invariant under the bar-involution.  Now apply the bar-involution to both sides of (\ref{fpoly-eqn-h}):
\begin{eqnarray} 
X_{j; t} = q^{-\lambda_{j; t}^{\tilde{B}^{0}, t_{0}}}M_{0}(\tilde{\gg}^{\tilde{B}^{0}_{j; t}})\overline{F_{j; t}(\hat{Y})} = q^{-\lambda_{j; t}^{\tilde{B}^{0}, t_{0}}}M_{0}(\tilde{\gg}^{\tilde{B}^{0}; t_{0}}_{j; t})\overline{F_{j; t}}(\hat{Y}). 
\end{eqnarray}
Equation (\ref{eq:fpoly-right}) follows from the fact that $\lambda_{j; t}^{\tilde{B}^{0}, t_{0}} = 0$ and $\tilde{\gg}_{j; t}^{\tilde{B}^{0}; t_{0}} = \gg_{j; t}$ when $\tilde{B}^{0}$ is principal. 
Applying the bar-involution to $F_{j; t}$ merely multiplies the coefficient of each monomial $Z^{\aa}$ in $F_{j; t}$ by $q^{c}$ for some $c \in \frac{1}{2}\ZZ$, so the last assertion is proven.
\end{proof}

For $\aa = (a_{1}, \ldots, a_{n}), \bb = (b_{1}, \ldots, b_{n}), \cc = (c_{1}, \ldots, c_{n}) \in \ZZ^{n}$, define the notation 
\begin{eqnarray}
\aa \cdot \bb \cdot \cc = \sum_{i = 1}^{n} a_{i}b_{i}c_{i} \in \ZZ.
\end{eqnarray}

We obtain the following result relating the $\gg$-vector $\gg_{j; t}$ and the coefficients of $F_{j; t}$.
 
\begin{corollary} \label{cor:fpoly-coeff} Let $j \in [1, n]$, $t \in \mathbb{T}_{n}$.  For $\aa \in \ZZ^{n}$,  let $P_{\aa} \in \ZZ[x, x^{-1}]$ such that $P_{\aa}(q^{\frac{1}{2}})$ is the coefficient of $Z^{\aa}$ in the quantum $F$-polynomial $F_{j; t}$.   Also, let $\dd = (d_{1}, \ldots, d_{n})$.

Then
\begin{eqnarray}
P_{\aa}(q^{\frac{1}{2}}) = q^{-\gg_{j; t} \cdot \aa \cdot \dd}P_{\aa}(q^{-\frac{1}{2}}).
\end{eqnarray}
In particular, if $P_{\aa}(q) = q^{c}$ for some $c \in \frac{1}{2}\ZZ$, then $c = -\frac{\gg_{j; t}\cdot \aa \cdot \dd}{2}$.
\end{corollary}
\begin{proof}
Using (\ref{eq:fpoly-right}) with (\ref{quasicomm2}), 
\begin{eqnarray}
M_{0}(\gg_{j; t})\overline{F}_{j; t}(\hat{Y}) & =  & \sum_{\aa \in \ZZ^{n}} M_{0}(\gg_{j; t})P_{\aa}(q^{-\frac{1}{2}})Z^{\aa} \\
& = & \left(\sum_{\aa \in \ZZ^{n}} q^{- \gg_{j; t} \cdot \aa \cdot \dd}P_{\aa}(q^{-\frac{1}{2}})Z^{\aa}\right) M_{0}(\gg_{j; t}).
\end{eqnarray}
The uniqueness of the quantum $F$-polynomial implies that 
\begin{eqnarray}
F_{j; t} = \sum_{\aa \in \ZZ^{n}} q^{-\gg_{j; t} \cdot \aa \cdot \dd}P_{\aa}(q^{-\frac{1}{2}})Z^{\aa}.
\end{eqnarray}
 The coefficient of $Z^{\aa}$ in this expression is equal to $q^{- \gg_{j; t} \cdot \aa \cdot \dd}P_{\aa}(q^{-\frac{1}{2}})$.
\end{proof}

The last result of this section is a recurrence relation for the quantum $F$-polynomials.   For the remainder of this section, assume that $\prinA = \prinA(B^{0}, D, \Lambda, t_{0})$ for some $n \times n$ skew-symmetric integer matrix $\Lambda$.  Let $\mathcal{R} = \mathbb{Z}[q^{\pm \frac{1}{2}}, Z_{1}^{\pm 1}, \ldots, Z_{n}^{\pm 1}]$.
For $\textbf{a} = (a_{1}, \ldots, a_{n}) \in \mathbb{Z}^{n}$, define a
$\ZZ[q^{\pm \frac{1}{2}}]$-linear operator $L[\textbf{a}]: \mathcal{R} \rightarrow \mathcal{R}$ by setting
\begin{eqnarray} \label{def:L-function}
L[\textbf{a}](Z^{\textbf{b}}) = q^{-\frac{1}{2}(\aa \cdot \bb \cdot \dd)}Z^{\textbf{b}},
\end{eqnarray}
for $\textbf{b} = (b_{1}, \ldots, b_{n}) \in \mathbb{Z}^{n}$.

The next lemma gives some basic properties of the operator $L[\aa]$.

\begin{lemma} Let $\aa, \cc \in \ZZ^{n}$, $F, G \in \mathcal{R}$.  Then 
\begin{enumerate}
\item[(1)] $F' = L[\aa](F)$ is the unique element of $\mathcal{R}$ such that
\begin{eqnarray} \label{L-eqn}
M_{0}(\aa)  F(\hat{Y})   = F'(\hat{Y}) M_{0}(\aa).
\end{eqnarray}
\item[(2)] $L[\aa + \cc] = L[\aa]\circ L[\cc] = L[\cc] \circ L[\aa]$
\item[(3)] $L[\aa](FG) = L[\aa](F)\cdot L[\aa](G)$
\end{enumerate}
\end{lemma}
\begin{proof}
For (1), use (\ref{quasicomm2}) to show first that (\ref{L-eqn}) holds when $F = Z^{\bb}$, then extend linearly.  Uniqueness follows from the fact that $\hat{Y}_{1}, \ldots, \hat{Y}_{n}$ are algebraically independent when $\tilde{B}^{0}$ is principal.  Statement (2) is a straightforward check.  For (3), let $H = L[\aa](FG)$, $F' = L[\aa](F)$, and $G' = L[\aa](G)$.  Applying part (1) two different ways gives 
\begin{eqnarray}
H(\hat{Y})M_{0}(\aa) = M_{0}(\aa)F(\hat{Y})G(\hat{Y}) = F'(\hat{Y})G'(\hat{Y})M_{0}(\aa).
\end{eqnarray}
Thus, $H(\hat{Y}) = F'(\hat{Y})G'(\hat{Y})$, which implies (3).
\end{proof}

Suppose $t \frac{k}{\hspace{1cm}} t'$ in $\mathbb{T}_{n}$.  Write $\tilde{B}^{t} = (b_{ij})$ for the exchange matrix at $t$.   For $j \in [1, n]$, $j \neq k$, $\epsilon \in \{ + , - \}$, let 
\begin{eqnarray}
G_{j}^{(\epsilon)} = \prod^{\rightarrow}_{i \in [1, [\epsilon b_{jk}]_{+}]} L[(i - 1)\gg_{j; t}](F_{j; t})
\end{eqnarray}
if $\epsilon b_{jk} \geq 1$; $G_{j}^{(\epsilon)} = 1$ otherwise.  Next, for $j \neq k$, let 
\begin{eqnarray}
\hat{F}_{j; t}^{\{ [\epsilon b_{jk}]_{+} \}} = L\left[-\gg_{k; t} + \displaystyle \sum_{i \in [1, j - 1]} [\epsilon b^{t}_{ik}]_{+} \gg_{i; t}\right](G_{j}^{(\epsilon)}).
\end{eqnarray}
Also, let
\begin{eqnarray}
\hat{F}_{k; t}^{\{-1\}} = (L[-\gg_{k; t}](F_{k; t}))^{-1}.
\end{eqnarray}
For $\epsilon \in \{ +, - \}$, let $\rho^{\epsilon}$ equal

\

$\frac{1}{2}\left(-\sum_{i = 1}^{2n} [\epsilon b_{ik}^{t}]_{+} \rho_{ik}^{t} + \sum_{1 \leq i < j \leq n} [\epsilon b^{t}_{ik}]_{+} [\epsilon b_{jk}^{t}]_{+} \rho_{ji}^{t} + \sum_{1 \leq i, j \leq n} [\epsilon b_{ik}^{t}]_{+} [\epsilon b_{j + n, k}^{t}]_{+} \rho_{n + j, i}^{t} \right).$

\

\noindent Finally, let 
\begin{eqnarray}
\lambda^{\epsilon} = \frac{1}{2}\sum_{i = 1}^{n} [\epsilon b_{n + i, k}^{t}]_{+} g_{i}' d_{i},
\end{eqnarray} 
where $\gg_{k; t'} = (g_{1}', \ldots, g_{n}')$.

\begin{theorem} \label{thm:quantum-fpoly-rec} The quantum $F$-polynomials $F_{j; t}$ are given by the following recurrence relations: 

The initial quantum $F$-polynomials are given by 
\begin{eqnarray} \label{fpoly-rec-init}
F_{j; t_{0}} = 1.
\end{eqnarray}

Suppose $t \frac{k}{\hspace{1cm}} t'$ in $\mathbb{T}_{n}$.  Then 
\begin{eqnarray} \label{fpoly-rec1}
F_{j; t'} = F_{j; t} \mbox{ if } j \neq k.
\end{eqnarray}

Using the notation given above, the quantum $F$-polynomial $F_{k; t'}$ is
\begin{eqnarray} \label{quantum-fpoly-rec}
F_{k; t'} & = & q^{(\rho^{+} - \lambda^{+})} \hat{F}_{k; t}^{\{-1\}} \left( \prod^{\rightarrow}_{i \in [1, n]} \hat{F}_{i; t}^{ \{ [b_{ik}]_{+} \} } \right) Z^{(\sum_{\ell = 1}^{n} [b_{n + \ell, k}]_{+} e_{\ell})} \\
\nonumber & & \hspace{1cm} + q^{(\rho^{-} - \lambda^{-})}\hat{F}_{k; t}^{\{-1\}} \left( \prod^{\rightarrow}_{i \in [1, n]} \hat{F}_{i; t}^{ \{ [-b_{ik}]_{+} \} } \right) Z^{(\sum_{\ell = 1}^{n} [- b_{n + \ell, k}]_{+} e_{\ell})}.
   \end{eqnarray}
\end{theorem}

\begin{remark} 
\begin{enumerate}
\item By Lemma \ref{lemma:rho-zero}, if the $F$-polynomial $F_{j; t}^{B^{0}; t_{0}}$ has nonzero constant term for every $j \in [1, n]$, $t \in \mathbb{T}_{n}$, then $\rho^{\epsilon} = 0$.
\item Setting $q = 1$ and $Z_{i} = u_{i}$ in $\hat{F}_{j; t}^{\{ r \}}$ yields $(F_{j; t}^{B^{0}; t_{0}})^{r}$; under this specialization, the quantum $F$-polynomial recurrence above becomes the $F$-polynomial recurrence given in Proposition \ref{fpoly-rec}.
\end{enumerate}
\end{remark} 

\begin{proof} (\ref{fpoly-rec-init}) and (\ref{fpoly-rec1}) are already known from the proof of Theorem \ref{thm:quantum-fpoly}.  Let $t \frac{k}{\hspace{1cm}} t'$ in $\mathbb{T}_{n}$.    The cluster variable $X_{k; t'} \in \prinA$ can be expressed as 
\begin{eqnarray}
X_{k; t'}  =  \displaystyle \sum_{\epsilon \in \{ + , - \}} M_{t}(-\textbf{e}_{k} + \sum_{i = 1}^{2n} [\epsilon b_{ik}]_{+} \textbf{e}_{i}). 
\end{eqnarray}
For $\epsilon \in \{ + , - \}$, let $\rho_{1}^{\epsilon} \in \frac{1}{2}\mathbb{Z}$ such that 
\begin{eqnarray}
& M_{t}(-\textbf{e}_{k} + \displaystyle \sum_{i = 1}^{2n} [\epsilon b_{ik}]_{+} \textbf{e}_{i}) \hspace{8cm} &  \\
\nonumber & \hspace{2cm} =  q^{\rho_{1}^{\epsilon}}M_{t}(-\textbf{e}_{k}) \displaystyle \left( \prod^{\rightarrow}_{i \in [1, n]} M_{t}([\epsilon b_{ik}]_{+} \textbf{e}_{i}) \right) M_{0}(\sum_{i = 1}^{n} [\epsilon b_{n + i, k}]_{+}\textbf{e}_{n + i}). &
\end{eqnarray}

\

Now $X_{k; t'}$ may be rewritten as 
\begin{eqnarray}
\displaystyle 
\sum_{\epsilon \in \{ + , - \}} q^{\rho_{1}^{\epsilon}}M_{t}(-\textbf{e}_{k})\left( \prod_{i \in [1, n]}^{\rightarrow} M_{t}([\epsilon b_{ik}]_{+} \textbf{e}_{i}) \right) M_{0}(\sum_{i = 1}^{n} [\epsilon b_{n + i, k}]_{+}\textbf{e}_{n + i}).
\end{eqnarray}

\

Using (\ref{fpoly-eqn-g}), $X_{k; t'}$ can be further rewritten as 
\begin{eqnarray*}
\displaystyle \sum_{\epsilon \in \{ + , - \}}  q^{\rho_{1}^{\epsilon}} (F_{k; t}(\hat{Y})M_{0}(\gg_{k; t}))^{-1} \left( \prod_{i \in [1, n]}^{\rightarrow} (F_{i;t}(\hat{Y})M_{0}(\gg_{i; t}))^{[\epsilon b_{ik}]_{+}}\right) M_{0}(\sum_{i = 1}^{n} [\epsilon b_{n + i, k}]_{+}\textbf{e}_{n + i}).  
\end{eqnarray*}
Let 
\begin{eqnarray}
\gg_{k; t}^{(\epsilon)} = -\gg_{k; t} + \sum_{i = 1}^{n} [\epsilon b_{ik}]_{+} \gg_{i; t} + \sum_{i = 1}^{n} [\epsilon b_{n + i, k}]_{+}\textbf{e}_{i + n}.
\end{eqnarray}

Pushing all the quantum $F$-polynomials to the left in the above expression for $X_{k; t'}$, the cluster variable may further be rewritten as 
\begin{eqnarray}
X_{k; t'} = \displaystyle  \sum_{\epsilon \in \{ + , - \}} q^{\rho^{\epsilon}} \hat{F}_{k; t}^{\{ -1 \}}(\hat{Y}) \left(\prod^{\rightarrow}_{i \in [1, n]} \hat{F}_{i; t}^{\{ [\epsilon b_{ik}]_{+} \}}(\hat{Y})\right)M_{0}(\gg_{k; t}^{(\epsilon)}).
\end{eqnarray}
From the recurrence relation for $\gg$-vectors (Proposition \ref{g-vec-rec-b}), it follows that 
\begin{eqnarray}
-\gg_{k; t'} + \gg_{k; t}^{(\epsilon)} = \sum_{i = 1}^{n} [\epsilon b_{n + i, k}]_{+} \tilde{\textbf{b}}^{i}    \hspace{1cm} (\epsilon \in \{ + , - \})
\end{eqnarray}
where $\tilde{\textbf{b}}^{i}$ is the $i$th column of $\tilde{B}^{0}$, the principal matrix with respect to $B^{0}$.  For $\epsilon \in \{ +, - \}$, observe that $\lambda^{\epsilon} \in \frac{1}{2}\mathbb{Z}$ satisfies
\begin{eqnarray}
\lambda^{\epsilon} = \frac{1}{2} \Lambda_{0}(-\gg_{k; t'}, \sum_{i = 1}^{n} [\epsilon b_{n + i, k}]_{+} \tilde{b}^{i} + \gg_{k; t'}), 
\end{eqnarray}
since $\Lambda_{0}(\ee_{j}, \tilde{\bb}^{i}) = -\delta_{ij} d_{i}$ for $i \in [1, 2n]$, $j \in [1, n]$.
This means that
\begin{eqnarray}
M_{0}(-\gg_{k; t'} + \gg_{k; t}^{(\epsilon)}) = q^{\lambda^{\epsilon}}M_{0}(\gg_{k; t}^{(\epsilon)})M_{0}(-\gg_{k; t'}).
\end{eqnarray}
Consequently, 
\begin{eqnarray}
& & F_{k; t'}(\hat{Y}) =  X_{k; t'}M_{0}(-\gg_{k; t'}) \\ 
\nonumber & & \hspace{1.4cm} =   \displaystyle  \sum_{\epsilon \in \{ + , - \}} q^{(\rho^{\epsilon} - \lambda^{\epsilon})} \hat{F}_{k; t}^{\{ -1 \}}(\hat{Y}) \left( \prod_{i \in [1, n]}^{\rightarrow} \hat{F}_{i; t}^{\{[\epsilon b_{ik}]_{+} \}}(\hat{Y}) \right) M_{0}(\sum_{i = 1}^{n} [\epsilon b_{n + i, k}]_{+} \tilde{\bb}^{i}).
\end{eqnarray}
The theorem follows the fact that $\hat{Y}_{1}, \ldots, \hat{Y}_{n}$ are algebraically independent.
\end{proof}

\begin{example}[\emph{Type~$A_2$; cf. \cite[Section~6]{ca1}}]  For $n = 2$, the tree $\TT_{2}$ is an infinite chain.  We call the vertices $\ldots, t_{-1}, t_{0}, t_{1}, t_{2}, \ldots$, and label the edges as follows:
\begin{eqnarray}
\label{eq:TT2}
\cdots
\overunder{2}{} t_{-1}
\overunder{1}{} t_0
\overunder{2}{} t_1
\overunder{1}{} t_2
\overunder{2}{} t_3
\overunder{1}{} \cdots \,.
\end{eqnarray}
Let $B^{0} = \left(  \begin{array}{cc} 0 & 1 \\ -1 & 0 \end{array} \right)$ be the initial $n \times n$ exchange matrix, and $D = \left(  \begin{array}{cc} 2 & 0 \\ 0 & 2 \end{array} \right)$.  Let $\tilde{B}^{0}$ be the principal matrix corresponding to $B^{0}$.  Then we may use matrix mutation (equation (\ref{eq:matrix-mutation})) and the $\gg$-vector recurrences (Proposition \ref{g-vec-rec-b}) to compute $\tilde{B}^{t}$ and the $\gg$-vectors $\gg_{j; t} = \gg_{j; t}^{B^{0}; t_{0}}$ which are given in Table \ref{table:A2-quantum-fpolys}.  Also, the quantum $F$-polynomials $F_{i; t_{0}}$ and $F_{i; t_{1}}$ ($i = 1, 2$) may be computed using Example \ref{example:simple-quantum-fpolys}.

\begin{table}[ht]
\begin{equation*}
\begin{array}{|c|c|cc|cc|}
\hline
&&&&&\\[-4mm]
t & \tilde B^{t} & \hspace{5mm} \gg_{1;t} & \gg_{2;t}& F_{1;t} & F_{2;t} \\[1mm]
\hline
&&&&&\\[-3mm]
0 &
\left[\begin{smallmatrix}0&1\\-1&0\\1&0\\0&1\end{smallmatrix}\right] &
\ee_1&\ee_2 &
1 & 1
\\[4.5mm]
\hline
&&&&&\\[-3mm]
1 & \left[\begin{smallmatrix}0&-1\\1&0\\1&0\\0&-1\end{smallmatrix}\right] &
\ee_1 & -\ee_{2} &
1& qZ_2+1
\\[4.5mm]
\hline
&&&&&\\[-3mm]
2 & \left[\begin{smallmatrix}0&1\\-1&0\\-1&0\\0&-1\end{smallmatrix}\right] &
-\ee_{1} &
-\ee_{2} &
qZ^{\ee_{1} + \ee_{2}} + qZ_1+1 &
qZ_2+1
\\[4.5mm]
\hline
&&&&&\\[-3mm]
3 & \left[\begin{smallmatrix}0&-1\\1&0\\-1&0\\-1&1\end{smallmatrix}\right] &
-\ee_{1}&
\ee_{2} - \ee_{1}&
qZ^{\ee_{1} + \ee_{2}}  + qZ_{1} + 1 &
qZ_1+1
\\[4.5mm]
\hline
&&&&&\\[-3mm]
4 & \left[\begin{smallmatrix}0&1\\-1&0\\1&-1\\1&0\end{smallmatrix}\right] &
\ee_{2} &
\hspace{2mm} \ee_{2} - \ee_{1}&
1 &
qZ_1+1
\\[4.5mm]
\hline
&&&&&\\[-3mm]
5 &
\left[\begin{smallmatrix}0&-1\\1&0\\0&1\\1&0\end{smallmatrix}\right] &
 \ee_2& \ee_1 & 1 & 1
\\[4.5mm]
\hline
\end{array}
\end{equation*}
\smallskip

\caption{Type~$A_2$, quantum $F$-polynomials}
\label{table:A2-quantum-fpolys}
\end{table}

For the remaining quantum $F$-polynomials, we use recurrence relations given in Theorem \ref{thm:quantum-fpoly-rec}.  To compute $F_{1; t_{2}}$, let $t = t_{1}$, $t' = t_{2}$, $k = 1$.  In this case,   $\hat{F}^{\{-1\}}_{1; t_{1}} = 1$,  $\hat{F}_{2; t_{1}}^{\{1\}} = qZ_{2} + 1$, and the recurrence (\ref{quantum-fpoly-rec}) becomes
\begin{eqnarray*}
F_{1; t_{2}} = q\hat{F}_{2; t_{1}}^{\{1\}} Z_{1} + 1 = q(qZ_{2} + 1)Z_{1} + 1 = qZ^{\ee_{1} + \ee_{2}} + qZ_{1} + 1. 
\end{eqnarray*}
To compute $F_{2; t_{3}}$, let $t = t_{2}, t' = t_{3}, k = 2$.  Then $\hat{F}_{2; t_{2}}^{\{-1\}} = (q^{-1}Z_{2} + 1)^{-1}$, $\hat{F}_{1; t_{2}}^{\{1\}}  = q^{-1}Z^{\ee_{1} + \ee_{2}} + qZ_{1} + 1$.  The recurrence (\ref{quantum-fpoly-rec}) in this case yields
\begin{eqnarray*}
F_{2; t_{3}} =  \hat{F}_{2; t_{2}}^{\{-1\}}( \hat{F}_{1; t_{2}}^{\{1\}} + q^{-1}Z_{2})  =  qZ_{1} + 1.
\end{eqnarray*}
For $F_{1; t_{4}}$, let $t = t_{3}$, $t' = t_{4}$, $k = 1$.   Then $\hat{F}_{1; t_{3}}^{\{-1\}} = (q^{-1}Z^{\ee_{1} + \ee_{2}} + q^{-1}Z_{1} + 1)^{-1}$, $\hat{F}_{2; t_{3}}^{\{1\}} = q^{-1}Z_{1} + 1$.  The recurrence in this case is 
\begin{eqnarray}
F_{1; t_{4}} = \hat{F}_{1; t_{3}}^{\{-1\}}(\hat{F}_{2; t_{3}}^{\{1\}} + q^{-1}Z^{\ee_{1} + \ee_{2}}) = 1.
\end{eqnarray}
Finally, for $F_{2; t_{5}}$, let $t = t_{4}$, $t' = t_{5}$, $k = 2$.   Then $\hat{F}_{2; t_{4}}^{\{-1\}} = (q^{-1}Z_{1} + 1)^{-1}$, $\hat{F}_{1; t_{4}}^{\{1\}} = 1$, and the recurrence gives
\begin{eqnarray}
F_{2; t_{5}} = \hat{F}_{2; t_{4}}^{\{-1\}}(\hat{F}_{1; t_{4}}^{\{1\}}  + q^{-1}Z_{1}) = 1. 
\end{eqnarray}

\end{example}

\section{Examples of Quantum $F$-polynomials} \label{section:examples}

For any $n \times n$ integer skew-symmetric matrix $B = (b_{ij})$, we may define a quiver $Q(B)$ on the set of vertices $[1, n]$, with $|b_{ij}|$ arrows from $i$ to $j$ if and only if $b_{ij} < 0$.    Let $B^{0}$ be any $n \times n$ skew-symmetric integer matrix, and let $D = dI_{n}$, where $I_{n}$ is the $n \times n$ identity matrix and $d$ is a positive integer.     Let $Q^{0} = Q(B^{0})$, and fix $T \subset [1, n]$ such that the subgraph of $Q^{0}$ induced by $T$ is a tree.  In particular,  there is at most one edge in $Q^{0}$ connecting any given pair of  vertices in $T$.   Without loss of generality, assume that $T = [1, \ell] \subset [1, n]$.  Furthermore, we may assume that the vertices of $T$ are labeled so that for each $i \in [1, \ell]$, the subgraph of $T$ induced by the set $[1, i]$ is also a tree, and in that tree, the vertex $i$ is a leaf.  

In this section, we associate to $T$ a cluster variable in $\prinA(B^{0}, t_{0})$ and compute the corresponding quantum $F$-polynomial and $\gg$-vector.  The identification of $T$ to such a cluster variable is done via \emph{denominator vectors}.   We recall from \cite{ca1} the definition of these vectors and the recurrence relations from which they may be computed.  Let $\mathcal{A}$ be any cluster algebra whose initial exchange matrix has principal part $B^{0}$.  Suppose $\mathcal{A}$ has initial extended cluster $(x_{1}, \ldots, x_{m})$.  By the Laurent phenomenon, any cluster variable $x_{j; t} \in \mathcal{A}$ may be expressed as 
\begin{eqnarray}
x_{j; t} = \displaystyle \frac{N(x_{1}, \ldots, x_{n})}{x_{1}^{d_{1}} \ldots x_{n}^{d_{n}}  },
\end{eqnarray}
where $N(x_{1}, \ldots, x_{n})$ is a polynomial with coefficients in $\ZZ[x^{\pm 1}_{n + 1}, \ldots, x^{\pm 1}_{m}]$ which is not divisible by any $x_{i}$.  Let $\dd_{j; t}^{B^{0}; t_{0}} = \left[   \begin{array}{c} d_{1} \\ \vdots \\ d_{n}   \end{array}  \right]$, and call this the \emph{denominator vector} of the cluster variable $x_{j; t}$.    Similarly, denominator vectors may be defined in quantum cluster algebras.

The vectors $\dd_{j;t} = \dd_{j;t}^{B^0;t_0}$ are uniquely
determined by the initial conditions
\begin{equation}
\label{eq:denom-cluster-variable-initial}
\dd_{j;t_0}^{B^0;t_0} = -\ee_j
\end{equation}
(where $\ee_1, \dots, \ee_n$ are the standard basis vectors
in~$\ZZ^n$) together with the recurrence relations implied by
the exchange relation~\eqref{eq:exchange-relation}: for $t \overunder{k}{} t'$ in $\mathbb{T}_{n}$, we have 
\begin{equation}
\label{eq:exchange-denominator}
\dd_{j;t'} =
\begin{cases}
\dd_{j;t} & \text{if $j\neq k$;}\\
- \dd_{k;t} +
\max\Bigl(\displaystyle\sum_{i = 1}^{n} [b_{ik}^t]_+ \dd_{i;t},
\sum_{i = 1}^{n} [-b_{ik}^t]_+ \dd_{i;t}\Bigr)
& \text{if $j=k$}
\end{cases}
\end{equation}
Here, the $\max$ operation on vectors is performed component-wise.  Observe that the denominator vector depends on $B^{0}, t_{0}, t, j$, not on the choice of coefficients or quantization.

For $S \subset [1, n]$, write 
\begin{eqnarray}
\ee_{S} = \sum_{i \in S} \ee_{i} \in \ZZ^{n}.
\end{eqnarray}
The next two propositions are an immediate consequence of Proposition 5.7 and Corollary 5.8 of \cite{qps2}: 

\begin{proposition} \label{prop:cluster-vars-trees} There exists a cluster variable $x_{T}$ in $\prinA(B^{0}, t_{0})$ such that the denominator vector of $x_{T}$  is $\ee_{T}$.   Furthermore, this cluster variable may be obtained from the initial cluster by mutating in directions $k = 1, \ldots, \ell$ and taking the $\ell$th cluster variable in the resulting cluster.
\end{proposition}

Let $t_{1}, \ldots, t_{\ell}$ in $\mathbb{T}_{n}$ such that $t_{i - 1} \frac{i}{\hspace{1cm}} t_{i}$ in $\mathbb{T}_{n}$ for $i = 1, \ldots, \ell$.  Then the proposition implies that the cluster variable $x_{i; t_{i}} \in \prinA(B^{0}, t_{0})$ is equal to $x_{[1, i]}$ for each $i$.   Write $F_{T} = F_{\ell; t_{\ell}}^{B^{0}; D; t_{0}}$, $F^{cl}_{T} = F^{B^{0}; t_{0}}_{\ell; t_{\ell}}$, and $\gg_{T} = \gg_{\ell; t_{\ell}}^{B^{0}; t_{0}}$ for the quantum $F$-polynomial, nonquantum (or "classical") $F$-polynomial, and $\gg$-vector corresponding to $T$, respectively.   

\begin{proposition} \label{fpoly-trees} \cite[Corollary 5.8]{qps2} The $F$-polynomial corresponding to $T$ is given by 
\begin{eqnarray}
F^{cl}_{T}(u_{1}, \ldots, u_{n}) = \sum_{S} \prod_{i \in S} u_{i}
\end{eqnarray}
where the summation ranges over all subsets $S \subset T$ such that the following condition holds:
\begin{eqnarray} \label{subsetprop0}
\mbox{if }j \in S \mbox{ and } i \in T \mbox{ such that }j \rightarrow i \mbox{ in }Q^{0}, \mbox{ then } i \in S.
\end{eqnarray}
\end{proposition}

\

For $k \in [1, n]$, define
\begin{eqnarray} \label{def:I-out-in}
I_{\mbox{\tiny{in}}}(k) = \{ i \in T : i \rightarrow k  \mbox{ in } Q^{0}\}, &  I_{\mbox{\tiny{out}}}(k) = \{ j \in T : k \rightarrow j \mbox{ in } Q^{0} \}.
\end{eqnarray}

\begin{proposition} \cite[Remark 5.9]{qps2}  \label{thm:gvec-trees} For $k \in T$, the $k$th component of $\gg_{T}$ is equal to 
\begin{eqnarray}
g_{k} = | I_{\rm out}(k)  |- 1.
\end{eqnarray}
\end{proposition}

A formula for the remaining components of the $\gg$-vector will be given in Proposition \ref{thm:gvec-trees2} in a particular situation.

For $S \subset [1, n]$, let $\phi(S)$ be the number of components in the subgraph of $Q^{0}$ induced by $S$.   

\begin{theorem} \label{thm:quantum-fpoly-trees} The quantum $F$-polynomial $F_{T}$ is given by 
\begin{eqnarray}
F_{T} = \sum q^{\frac{d}{2}\phi(S)}Z^{\ee_{S}}
\end{eqnarray}
where the summation ranges over subsets $S \subset T$ such that (\ref{subsetprop0}) is satisfied.
\end{theorem}
\begin{proof}
First, we need to show that the quantum $F$-polynomial $F_{T}$ is given by 
\begin{eqnarray} \label{eqn:trees-quantum-fpoly}
F_{T} = \sum_{S} q^{-\frac{d}{2}(\gg_{T}\cdot \ee_{S})} Z^{\ee_{S}},
\end{eqnarray} 
where the summation ranges over $S \subset T$ satisfying (\ref{subsetprop0}).

Apply induction on $\ell$, the number of vertices in $T$.   If $T = \{ 1 \}$, then Proposition \ref{thm:gvec-trees} implies that the first component of $\gg_{T}$ is $-1$.  Since $F_{T} = F^{B^{0}; D; t_{0}}_{1; t_{1}}$, (\ref{eqn:trees-quantum-fpoly}) follows from Example \ref{example:simple-quantum-fpolys}.  Next, assume that (\ref{eqn:trees-quantum-fpoly}) is known for $F_{[1, i]}$, where $i = 1, \ldots, \ell - 1$.   It suffices to prove that 
\begin{eqnarray} \label{eqn:fpoly-trees-first}
F_{T} = \sum_{S} P_{S}(q^{\frac{1}{2}}) Z^{\ee_{S}},
\end{eqnarray}
where the summation ranges over $S \subset T$ such that (\ref{subsetprop0}) holds, and each $P_{S}(q^{\frac{1}{2}})$ is of the form $q^{\lambda}$ for some $\lambda \in \frac{1}{2}\ZZ$.  Then (\ref{eqn:trees-quantum-fpoly})  follows from Corollary \ref{cor:fpoly-coeff} and induction.

 Using $t = t_{\ell - 1}$ and $t' = t_{\ell}$, Theorem \ref{thm:quantum-fpoly-rec} implies that there exists some $\aa, \aa' \in \ZZ^{n}$ and $\lambda, \lambda' \in \frac{1}{2}\ZZ$ such that
\begin{eqnarray}  \label{rec-for-trees}
F_{T} = q^{\lambda}\hat{F}_{1}\ldots \hat{F}_{r}Z^{\aa} + q^{\lambda'}\hat{F}_{1}'\ldots \hat{F}_{s}' Z^{\aa'},
\end{eqnarray}
where the $\hat{F}_{j}$ and $\hat{F}_{k}'$ are each of the form $L[\hh](F_{[1, p]})$ for some $\hh \in \ZZ^{n}$, $p \in [1, \ell]$.  By (\ref{def:L-function}) and the induction hypothesis, the coefficients of any monomial $Z^{\cc}$ ($\cc \in \ZZ^{n}$) in the  $\hat{F}_{j}$ and $\hat{F}_{k}'$ are powers of $q$. Thus, (\ref{rec-for-trees}) implies that $F_{T} = \sum P_{S}(q^{\frac{1}{2}})Z^{\ee_{S}}$, where the sum ranges over $S \subset [1, n]$, and each $P_{S} \in \ZZ[x, x^{-1}]$ is a subtraction-free Laurent polynomial.

Now set $q = 1$ and $Z_{i} = u_{i}$ for all $i \in [1, n]$ in the equation $F_{T} = \sum P_{S}(q^{\frac{1}{2}})Z^{\ee_{S}}$.  By the last part of Theorem \ref{thm:quantum-fpoly}, this gives an expression for $F_{T}^{cl}$.   Note that no monomial in $F_{T}$ with nonzero coefficient disappears under this specialization.   By Proposition \ref{fpoly-trees}, it follows that $Z^{\ee_{S}}$ occurs with nonzero coefficient in $F_{T}$ if and only if $S \subset T$ and  (\ref{subsetprop0}) is satisfied; furthermore, in this case, we have $P_{S}(1) = 1$, which forces $P_{S}(q^{\frac{1}{2}})$ to be a power of $q$, as desired.
This finishes the proof of (\ref{eqn:trees-quantum-fpoly}).

To conclude the proof of the theorem, we need to prove that
\begin{eqnarray} \label{eqn:components}
-\gg_{T} \cdot \ee_{S} = \phi(S)
\end{eqnarray} 
for any $S \subset T$ such that $S$ satisfies  (\ref{subsetprop0}).  Fix such a subset $S$.  By Proposition \ref{thm:gvec-trees}, 
\begin{eqnarray} \label{eqn:g-dot}
-\gg_{T} \cdot \ee_{S} = \sum_{k \in S} (1- |I_{\rm out}(k)|) = |S| - \sum_{k \in S} |I_{\rm out}(k)|.
\end{eqnarray}
Since (\ref{subsetprop0}) holds for $S$,  we have that for $k \in S$, $|I_{\rm out}(k)|$ is the number of $i \in S$ such that $k \rightarrow i$ in $Q^{0}$.    Thus, $\sum_{k \in S} |I_{\rm out}(k)|$ is equal to the number of edges whose endpoints are both in $S$, which is also equal to the number of edges in the subgraph of $Q^{0}$ induced by $S$.  The number of vertices in a tree minus the number of edges is 1, so $|S| - \sum_{k \in S} |I_{\rm out}(k)|$ must equal the number of components in the subgraph of $Q^{0}$ induced by $S$.  \end{proof}

Next, we state a formula for $\gg$-vectors in a particular case.   This formula was given in \cite{qps2} in terms of representations of quivers with potentials. Let $B^{0}$ be an $n \times n$ skew-symmetric matrix with entries from $\{ 0, 1, -1 \}$.      We recall some notation and definitions from \cite{qps2}, omitting certain technical details which are not needed here.   Let $M$ be the quiver representation of $Q^{0}$ such that $M(i) = \CC$ if $i \in T$, and $M(i) = 0$ otherwise.  Also, for the arrow $a: i \rightarrow j$, let $a_{M}: M(i) \rightarrow M(j)$ be the  identity map if  $i, j \in T$, and let $a_{M}$ be the 0-map otherwise.

For $k \in [1, n]$, let 
\begin{eqnarray}
M_{\mbox{\tiny{in}}}(k) = \bigoplus_{i \in I_{\mbox{\tiny{in}}}(k)}  M(i), & M_{\mbox{\tiny{out}}}(k) = \displaystyle \bigoplus_{j \in I_{\mbox{\tiny{out}}}(k)}  M(j),
\end{eqnarray}
where $I_{\rm in}(k)$, $I_{\rm out}(k)$ were defined at (\ref{def:I-out-in}).

A certain linear map $\gamma_{k}: M_{\mbox{\tiny{out}}}(k) \rightarrow M_{\mbox{\tiny{in}}}(k)$ was defined in \cite{qps2}.  This map is given by the matrix $\gamma_{k} = (\gamma_{k}(i, j))$, where the rows of the matrix are indexed by $i \in I_{\mbox{\tiny{in}}}(k)$, the columns are indexed by $j \in I_{\mbox{\tiny{out}}}(k) $, and $\gamma_{k}(i, j)$ is either a nonzero element of $\CC$ if there exists a directed path from $j$ to $i$ in $Q^{0}$ with all vertices contained in $T$, and $\gamma_{k}(i, j) = 0$ otherwise.

\begin{lemma}  \label{lemma:gamma-k}
Let $k \in [1, n]$.  The rank of $\gamma_{k}$ is equal to the maximum number $r \in \ZZ_{\geq 0}$ for which there exist distinct $j_{1}, \ldots, j_{r} \in I_{\rm out}(k)$ and distinct $i_{1}, \ldots, i_{r} \in I_{\rm in}(k)$ such that there is a directed path from $j_{s}$ to $i_{s}$ in $T$ for all $s \in [1, r]$.    (In particular, $\rank(\gamma_{k})$ does not depend on the specific nonzero values in the matrix $\gamma_{k}$.)
\end{lemma}
\begin{proof} If $r = 0$, then $\gamma_{k} = 0$, and the lemma holds.  Assume for the remainder of the proof that $r \geq 1$.    Let $j_{1}, \ldots, j_{r}, i_{1}, \ldots, i_{r}$ be vertices as in the statement of the lemma.   

We claim that there does not exist $\sigma$ in the symmetric group $S_{r}$, $\sigma \neq \mbox{id}$, such that $\gamma_k(i_{\sigma(s)}, j_{s}) \neq 0$ for all $s \in [1, r]$.  For the sake of contradiction, assume that such a $\sigma$ exists.  Let $p$ be the order of $\sigma$ in $S_{r}$.  Then there exists a directed path from $j_{s}$ to $i_{\sigma(s)}$ in $T$ for all $s \in [1, r]$, so it follows that there is a cycle in $T$ containing the vertices $i_{1}, j_{1}, i_{\sigma(1)}, j_{\sigma(1)}, \ldots, j_{\sigma^{p - 1}(1)}, i_{\sigma^{p}(1)} = i_{1}$, which contradicts that fact that the subgraph of $Q^{0}$ induced by $T$ is a tree.

Let $\Gamma$ be the $r \times r$ submatrix of $\gamma_{k}$ with rows indexed by $i_{1}, \ldots, i_{r}$, and columns indexed by $j_{1}, \ldots, j_{r}$.    Then the determinant of  $\Gamma$ is $\prod_{s = 1}^{r} \gamma_{k}(i_{s}, j_{s})$, which is nonzero since $\gamma_{k}(i_{s}, j_{s}) \neq 0$ for all $s \in [1, r]$.  This proves that $\rank(\gamma_{k}) \geq r$.

Now let $r' = \rank(\gamma_{k})$, and suppose that $\Gamma'$ is an $r' \times r'$ invertible submatrix of $\gamma_{k}$ with rows indexed by $i_{1}', \ldots, i_{r'}' \in I_{\rm in}(k)$ and columns indexed by $j_{1}', \ldots, j_{r'}' \in I_{\rm out}(k)$.  In order for the determinant of $\Gamma'$ to be nonzero, there must exist $\sigma \in S_{r'}$ such that $\gamma_{i_{\sigma(s)}, j_{s}} \neq 0$ for all $s \in [1, r']$.   Thus, there is a directed path from $j_{s}$ to $i_{\sigma(s)}$ in $T$ for each $s \in [1, r']$, which means that $\rank(\gamma_{k}) \leq  r$.
\end{proof}

\begin{proposition} \label{thm:gvec-trees2} The $\gg$-vector $\gg_{T}$ corresponding to $T$ is given by $(g_{1}, \ldots, g_{n})$, where 
\begin{eqnarray}
g_{k} & = & \dim(\ker(\gamma_{k})) - \dim(M(k)) \\
& = & |  I_{\mbox{\tiny{out}}}(k)  | - \rank(\gamma_{k}) - \dim(M(k)).
\end{eqnarray}
for all $k \in [1, n]$.
\end{proposition}
\begin{proof}
The proposition follows immediately from Theorem 5.1 and Proposition 5.7 of \cite{qps2}.  
\end{proof}

\subsection{Type $\mbox{A}_{n}$}

In this subsection, assume that $B^{0} = (b^{0}_{ij})$ is an $n \times n$ exchange matrix of type $\mbox{A}_{n}$.    Recall that such a matrix $B^{0}$ may be obtained via a sequence of matrix mutations from the $n \times n$ matrix $B = (b_{ij})$, where 
\begin{eqnarray}
b_{ij} = \left\{  \begin{array}{rl} 1 & \mbox{ if } j = i + 1 \\
                                                        -1 & \mbox{ if } j = i - 1 \\
                                                        0 & \mbox{ otherwise} 
                                                        \end{array}   \right.
\end{eqnarray}

We will show that each cluster variable in the cluster algebra $\prinA(B^{0}, t_{0})$ corresponds to an induced chain $C$, which means that Theorem \ref{thm:quantum-fpoly-trees} may be used to compute all quantum $F$-polynomials in the type $\mbox{A}_{n}$ case.  Also, we compute  $\gg$-vectors for type $\mbox{A}_{n}$ using Proposition \ref{thm:gvec-trees2}.  

Denote by $\Phi_{+}(B^{0})$ the collection of subsets $C$ of $[1, n]$ such that the subgraph of $Q^{0}$ induced by $C$ is a chain.   (A \emph{chain} is an alternating sequence $v_{1}, e_{1}, v_{2},  \ldots, e_{p}, v_{p}$ of distinct vertices and edges such that the edge $e_{i}$ has vertices $v_{i}, v_{i + 1}$ for $i = 1, \ldots, p - 1$, and there are no other edges connecting the vertices $v_{1}, \ldots, v_{p}$.)

\begin{proposition} \label{typeAclustervars} The cluster variables in $\mathcal{A}$ which are not in the initial cluster are in bijective correspondence with the elements of $\Phi_{+}(B^{0})$.  To be more specific, if a cluster variable corresponds to the set $C \in \Phi_{+}(B^{0})$, then its denominator vector is $\ee_{C} = \sum_{i \in C} \ee_{i}$. 
\end{proposition}

\begin{remark}
For the acyclic case in finite type, it was proven in \cite{ca2} that the cluster variables not in the initial cluster are in bijective correspondence with the positive roots of the given type.  In the acyclic case, Proposition \ref{typeAclustervars} is a consequence of this result.  More recently, Proposition \ref{typeAclustervars} was independently stated and proven in \cite{mw}.
\end{remark} 

To prove the proposition, we will use the following combinatorial description of cluster algebras of type $\mbox{A}_{n}$ given in \cite{ca2}.  In this realization, the cluster variables are in bijective correspondence with the diagonals of the $(n + 3)$-gon $\mathbb{P}_{n + 3}$, and the clusters correspond to maximal sets of noncrossing diagonals, i.e., to triangulations of $\mathbb{P}_{n + 3}$.  (Note that two diagonals cross if they intersect in an interior point.) Let $\mathcal{T} = ( \beta_{1}, \ldots, \beta_{n})$ be a list of pairwise distinct noncrossing diagonals of $\mathbb{P}_{n + 3}$.  Write $\Delta(\mathcal{T})$ for the corresponding set of triangles.  In this setting, cluster mutations are encoded as follows.   Let $k \in [1, n]$.  To mutate $\mathcal{T}$ in direction $k$, let $\Delta_{1}, \Delta_{2}$ be the two triangles in $\Delta(\mathcal{T})$ which have $\beta_{k}$ as a side, and let $a, b$ be the vertices opposite the side $\beta_{k}$ in each of these triangles.  Then $\mu_{k}(\mathcal{T})$ is the list of diagonals obtained from $\mathcal{T}$ by replacing $\beta_{k}$ by the diagonal $\overline{ab}$.

One may associate to $\mathcal{T}$ a quiver $Q(\mathcal{T})$ on the set of vertices $[1, n]$.  For the edges, let $i, j \in [1, n]$, $i \neq j$.  If  there is no triangle in $\mathcal{T}$ such that $\beta_{i}$, $\beta_{j}$ are sides of the triangle, then there is no edge between $i$ and $j$.  Otherwise, suppose that the endpoints of $\beta_{i}$ are $a, b$, and the endpoints of $\beta_{j}$ are $a, c$.  Then $i \rightarrow j$ if $a, b, c$ are in clockwise order, and $i \leftarrow j$ if $a, b, c$ are in counterclockwise order.  The principal part of the exchange matrix $B(\mathcal{T}) = (b_{ij})$ corresponding to $\mathcal{T}$ is given by $b_{ij} = 0$ if there is no edge between $i$ and $j$, $b_{ij} = -1$ if $i \rightarrow j$, and $b_{ij} = 1$ if $i \leftarrow j$.    Note that $Q(B(\mathcal{T})) = Q(\mathcal{T})$.



In \cite{bv}, Buan and Vatne characterize all type $\mbox{A}_{n}$ quivers (i.e., all quivers $Q$ of the form $Q = Q(\mathcal{T})$ where $\mathcal{T}$ is a triangulation of $\mathbb{P}_{n + 3}$).    

\begin{lemma} \cite[Proposition 2.4]{bv}\label{quiverlemma1} Let $Q$ be a quiver with vertex set $[1, n]$.   Then $Q$ is of type $\mbox{A}_{n}$ if and only (1)-(4) hold:
\begin{enumerate}
\item Any induced cycle in $Q$ is an oriented 3-cycle.  (In particular, there are no multiple edges.)
\item The degree of any vertex is at most 4.
\item If a vertex $i$ has degree 4, then two of the edges containing $i$ are in a 3-cycle, and the other two edges containing $i$ are in another 3-cycle.
\item If a vertex $i$ has degree 3, then two of the edges containing $i$ are in a 3-cycle, and the other edge does not belong to any 3-cycle.
\end{enumerate}
In particular, it follows that any induced tree in $Q$ must be a chain.
\end{lemma}

Let $\mathcal{T}^{0} = \{ \alpha_{1}, \ldots, \alpha_{n} \}$ be the triangulation of $\mathbb{P}_{n + 3}$ corresponding to the cluster at $t_{0}$.   Proposition \ref{typeAclustervars} is an immediate consequence of the following lemma:

\begin{lemma}  \label{lemma:diagonals} The diagonals of $\mathbb{P}_{n + 3}$ which are not in $\mathcal{T}^{0}$ are in bijective correspondence with the elements of $\Phi_{+}(B^{0})$.  Under this correspondence, if $\beta$ is a diagonal not in $\mathcal{T}^{0}$, then the corresponding subset of $[1, n]$ consists precisely of those $i \in [1, n]$ for which $\beta$ crosses $\alpha_{i}$.    Furthermore, if $C = \{ p_{1}, \ldots, p_{j} \} \in \Phi_{+}(B^{0})$ such that there is some edge between $p_{i}$ and $p_{i + 1}$ for $i = 1, \ldots, j - 1$, then the cluster variable corresponding to $C$ can be obtained from the initial cluster by mutating in directions $p_{1}, \ldots, p_{j}$.  The denominator vector of this cluster variable is $\ee_{C}$.
\end{lemma} 

\begin{proof}  First, consider a diagonal $\beta$ of $\mathbb{P}_{n + 3}$ that is not in $\mathcal{T}^{0}$.  Let $\mathcal{T}'$ be the set of diagonals in $\mathcal{T}^{0}$ which $\beta$ intersects.  This set $\mathcal{T}'$ may be constructed as follows: Start with one endpoint of $\beta$.  This endpoint is the vertex of a triangle $\Delta_{1}$ in the initial triangulation such that $\beta$ passes through the interior of the triangle.   The diagonal $\beta$ crosses another diagonal $\alpha_{p_{1}}$ which is a side of the triangle $\Delta_{1}$.  The diagonal $\alpha_{p_{1}}$ is a side of another triangle $\Delta_{2}$ in the initial triangulation.  Either $\beta$ intersects a vertex of $\Delta_{2}$, in which case $\mathcal{T}' = \{ \alpha_{p_{1}} \}$, or $\beta$ crosses another side $\alpha_{p_{2}}$ of $\Delta_{2}$.   In the latter case, $\alpha_{p_{2}}$ is the side of another triangle $\Delta_{3}$ in the initial triangulation, and it follows that either $\mathcal{T}' = \{\alpha_{p_{1}}, \alpha_{p_{2}}\}$, or that $\alpha$ crosses another side $\alpha_{p_{3}} \neq \alpha_{p_{2}}$ of $\Delta_{2}$.  Continuing this process, we get $\mathcal{T}' = \{ \alpha_{p_{1}}, \ldots, \alpha_{p_{j}} \}$, and triangles $\Delta_{1}, \ldots, \Delta_{j}$ such that for each $i = 1, \ldots, j - 1$, the diagonal $\alpha_{p_{i}}$ is a side of the triangles $\Delta_{i}$ and $\Delta_{i + 1}$.  It is clear that the subgraph of $Q^{0}$ induced by the vertices  $p_{1}, \ldots, p_{j}$ does not contain a cycle, since Lemma \ref{quiverlemma1} implies that the vertices in any induced cycle in $Q^{0}$ correspond to the diagonals in a triangle in $\mathcal{T}^{0}$, and $\beta$ cannot cross all of the sides of a triangle.  By  Lemma \ref{quiverlemma1}, the subgraph of $Q^{0}$ induced by $p_{1}, \ldots, p_{j}$ must be a chain.

Next, consider a sequence $p_{1}, \ldots, p_{j}$ of vertices from $[1, n]$ such that the subgraph of $Q^{0}$ induced by these vertices is a chain (with an edge between any two consecutive vertices in the list).  The goal is to find a diagonal $\beta$ of $\mathbb{P}_{n + 3}$ which crosses $\alpha_{p_{1}}, \ldots, \alpha_{p_{j}}$ and no other diagonals in $\mathcal{T}^{0}$.  If $j = 1$, then let $\Delta_{0}$, $\Delta_{1}$ be the triangles in $\Delta(\mathcal{T}^{0})$ which have $\alpha_{p_{1}}$ as a side.  For $j \geq 2$,  any two consecutive diagonals $\alpha_{p_{i}}, \alpha_{p_{i + 1}}$ are the sides of a triangle $\Delta_{i}$ in the initial triangulation; also, there exist triangles $\Delta_{0}$, $\Delta_{j}$ with $\Delta_{0} \neq \Delta_{1}$, $\Delta_{j} \neq \Delta_{j - 1}$, such that  $\alpha_{p_{1}}$ is a side of $\Delta_{0}$, and $\alpha_{p_{j}}$ is a side of $\Delta_{j}$.   Let $a_{0}$ be the vertex of $\Delta_{0}$ which is opposite the side $\alpha_{p_{1}}$.  For $i = 1, \ldots, j$, let $a_{i}$ be the vertex of $\Delta_{i}$ which is opposite the side $\alpha_{p_{i}}$.   

We claim that $\beta = \overline{a_{0}a_{j}}$ is the desired diagonal.   Observe that all of the triangles $\Delta_{0}, \ldots, \Delta_{j}$ are distinct; otherwise, there would be three diagonals from the list $\alpha_{p_{1}}, \ldots, \alpha_{p_{j}}$ as sides of a triangle, which would mean that the vertices $p_{1}, \ldots, p_{j}$ induce a cycle in $Q^{0}$.  By construction, the total set of vertices from the triangles $\Delta_{0}, \ldots, \Delta_{j}$  contains at most $j + 3$ vertices.  Any triangulation of $P =  \mbox{Conv}(\Delta_{0} \cup \ldots \cup \Delta_{j})$ has at most $j + 1$ triangles, so it follows that $\Delta_{0}, \ldots, \Delta_{j}$ is a triangulation of $\mbox{Conv}(\Delta_{0} \cup \ldots \cup \Delta_{j})$. Thus, $P = \Delta_{0} \cup \ldots \cup \Delta_{j}$ is convex.  This means that the only diagonals from $\mathcal{T}^{0}$ that $\beta$ can intersect are in the list $\alpha_{p_{1}}, \ldots, \alpha_{p_{j}}$.  To show that these are exactly the diagonals from $\mathcal{T}^{0}$ which $\beta$ intersects, consider the diagonals of $\mathcal{T}^{0}$ which $\beta$ passes through as one moves from the endpoint $a_{0}$ to the other endpoint $a_{j}$.  If $j = 1$, then $\beta$ intersects only the diagonal $\alpha_{p_{1}}$.  Otherwise, $\beta$ intersects the interior of $\Delta_{1}$, and passes through the interior of another side of $\Delta_{1}$ different from $\alpha_{p_{1}}$.  This side must be $\alpha_{p_{2}}$, since the side of $\Delta_{1}$ different from $\alpha_{p_{1}}$ and $\alpha_{p_{2}}$ is on the boundary of $P$.  Continuing this argument, it is easy to show that $\beta$ intersects $\alpha_{p_{1}}, \ldots, \alpha_{p_{j}}$.

Suppose that the triangulation $\mathcal{T}^{0}$ is mutated in directions $p_{1}, \ldots, p_{j}$, and call the resulting sequence of triangulations $\mathcal{T}^{1}, \ldots, \mathcal{T}^{j}$.  One verifies by induction that when $\mathcal{T}^{i - 1}$ is mutated to $\mathcal{T}^{i}$, the diagonal $\alpha_{p_{i}}$ is flipped to $\overline{a_{0}a_{i}}$.  The assertion about the denominator vector follows from Proposition \ref{prop:cluster-vars-trees}.  \end{proof}  

\begin{remark} The final assertion in Lemma \ref{lemma:diagonals} about the denominator vector is proven in greater generality for all triangulated surfaces in \cite[Theorem 8.6]{triangulated}.
\end{remark}

\begin{example} Let $\mathcal{T}^{0}$ be the triangulation of $\mathbb{P}_8$ given in Figure \ref{fig:T0}.   The quiver $Q^{0} = Q(\mathcal{T}^{0})$ of type $\mbox{A}_{5}$ is given below.

\hspace{5cm} \begin{xy} 0;<1pt,0pt>:<0pt,-1pt>:: 
(0,45) *+{1} ="0",
(45,45) *+{2} ="1",
(68,0) *+{3} ="2",
(92,45) *+{4} ="3",
(115,0) *+{5} ="4",
"0", {\ar"1"},
"1", {\ar"2"},
"3", {\ar"1"},
"2", {\ar"3"},
"4", {\ar"2"},
\end{xy}

\vspace{0.5cm}

Then the diagonals $\overline{af}$ and $\overline{dg}$ correspond to cluster variables in $\prinA(B^{0}, t_{0})$ with denominator vectors $\ee_{1} + \ee_{2} + \ee_{3} + \ee_{5}$ and $\ee_{3} + \ee_{4}$, respectively.

\begin{figure}
  \includegraphics[width=3in]{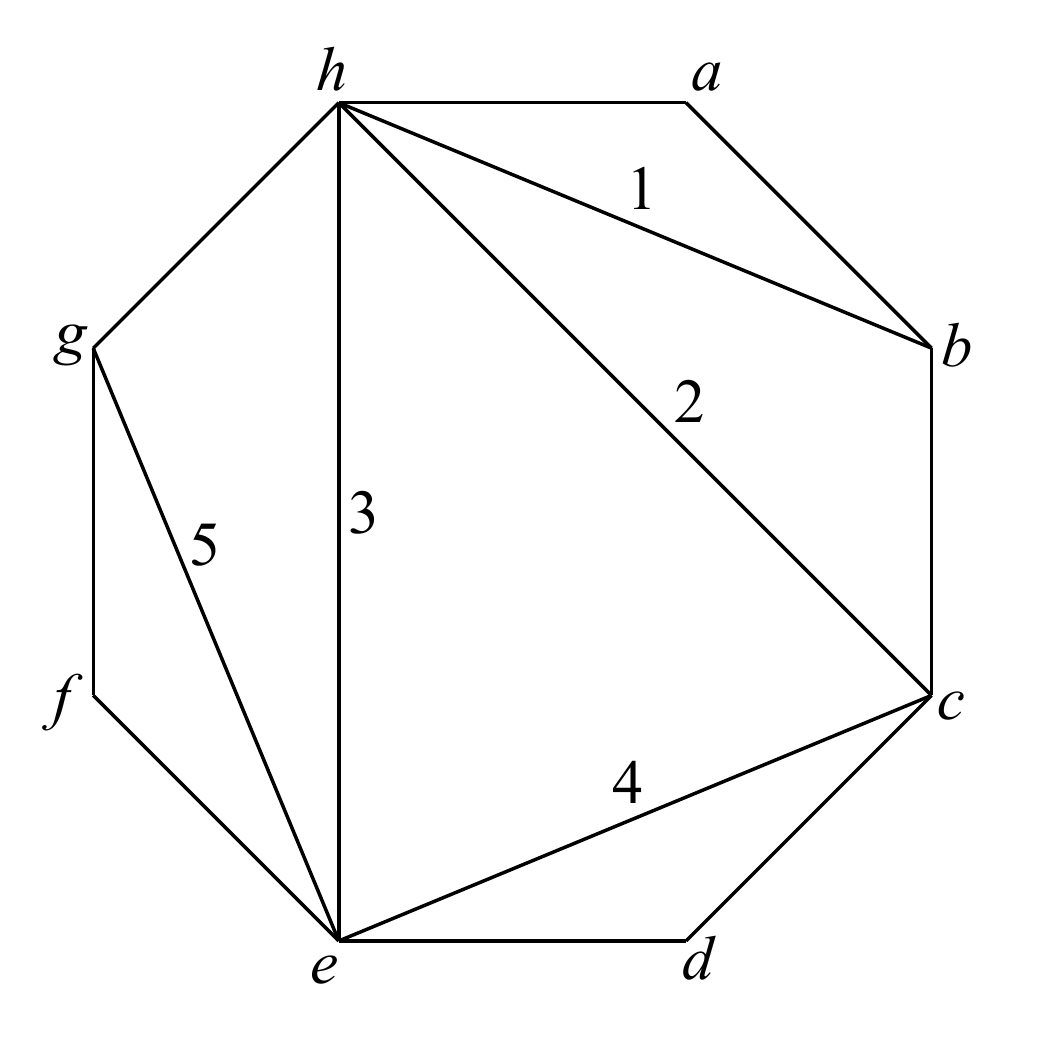}
  \caption{}
  \label{fig:T0}
\end{figure}
\end{example}

\begin{proposition} \label{typeA-gvec}The $\gg$-vector of the cluster variable corresponding to the set $C  \in \Phi_{+}(B^{0})$ is $(g_{1}, \ldots, g_{n})$, where 
\begin{itemize}
\item if $k \in C$, then $g_{k} = |  I_{\mbox{\tiny{out}}}(k) | - 1$;
\item if $k \notin C$, the subgraph of $Q^{0}$ induced by  $C \cup \{ k \}$ is a chain, and $k \rightarrow j$ in $Q^{0}$ for some $j \in C$, then $g_{k} = 1$;
\item otherwise, $g_{k} = 0$.
\end{itemize}
\end{proposition}

\begin{remark} Proposition \ref{typeA-gvec} was independently stated and proven in \cite{mw}, and generalized to other classical types.  For finite type, formulas for (nonquantum) $F$-polynomials and $\gg$-vectors were given in \cite{coefficients} in the bipartite case and in \cite{yangzel} for the acyclic case.
\end{remark}

\begin{proof} The first part follows from Proposition \ref{thm:gvec-trees}.  Using Proposition \ref{thm:gvec-trees2} and the notation preceding it (with $T$ replaced by $C$), it suffices to compute $g_{k} = \dim(\ker(\gamma_{k}))$ for $k \in [1, n] - C$.  Fix such an index $k$. 

First, assume that there exists $j \in C$ such that $k \rightarrow j$ in $Q^{0}$, and the subgraph of $Q^{0}$ induced by $C \cup \{k \}$ is a chain. Then $M_{\mbox{\tiny{in}}}(k) = 0$, so $\ker(\gamma_{k}) = M_{\mbox{\tiny{out}}}(k) = M(j)$, which means that $\dim(\ker(\gamma_{k})) = 1$.

Next, suppose that there exists $j \in C$ such that $k \rightarrow j$ in $Q^{0}$, and the subgraph of $Q^{0}$ induced by $C \cup \{k \}$ is not a chain (which means that it contains a cycle by Lemma \ref{quiverlemma1}).  By the same lemma, the only induced cycles in $Q^{0}$ are directed 3-cycles, so it follows that there exists $p \in C$ such that $k \rightarrow j \rightarrow p \rightarrow k$ in $Q^{0}$.   It may also be deduced from the same lemma that  there is no edge between $k$ and another vertex $j' \in C - \{ j, p \}$.  Thus, $M_{\rm out}(k) = M(j)$ and $M_{\rm in}(k) = M(p)$.  Since $\gamma_{k} $ is an isomorphism between $M(j)$ and $M(p)$,   it follows that  $\dim(\ker(\gamma_{k})) = 0$.

Finally, suppose that there is no $j \in C$ such that $k \rightarrow j$ in $Q^{0}$.  Then $M_{\mbox{\tiny{out}}}(k) = 0$, so $\ker(\gamma_{k}) = 0$.  
\end{proof}

\begin{example}
Let $B^{0}$ be the following initial exchange matrix of type $\mbox{A}_{4}$:
\begin{eqnarray}
B^{0} = \left( \begin{array}{rrrr}
0 & 1 & -1 & 0 \\
-1 & 0 & 1 & 0 \\
1 & -1 & 0 & -1 \\
0 & 0 & 1 & 0
\end{array} \right)
\end{eqnarray}
Then $Q^{0} = Q(B^{0})$ is the quiver below:

\hspace{5cm} \begin{xy} 0;<1pt,0pt>:<0pt,-1pt>:: 
(0,36) *+{1} ="0",
(25,0) *+{2} ="1",
(50,36) *+{3} ="2",
(100,36) *+{4} ="3",
"1", {\ar"0"},
"0", {\ar"2"},
"2", {\ar"1"},
"2", {\ar"3"},
\end{xy}

\vspace{0.5cm}
 Also, let $D = 2I_{4}$, where $I_{4}$ is the $4 \times 4$ identity matrix.  The complete list of denominator vectors corresponding to cluster variables not in the initial cluster is given in Table \ref{table:F-polys}.  The corresponding $\gg$-vectors and quantum $F$-polynomials are computed using Proposition \ref{typeA-gvec} and Theorem \ref{thm:quantum-fpoly-trees}.  To obtain the $F$-polynomial, plug $q = 1$ and $Z_{i} = u_{i}$ for $i \in [1, 4]$ into the corresponding quantum $F$-polynomial.

\end{example}

\begin{table}
\begin{equation*}
\begin{array}{|c|c|c|}
\hline
&&\\[-4mm]
\mbox{Denominator} & \gg\mbox{-vector}& \mbox{Quantum } \\
\mbox{vector} &    &  F\mbox{-polynomial} \\
\hline
&&\\[-4mm]
\ee_{1} & -\ee_{1} + \ee_{2} &  qZ_{1} + 1 \\[1mm]
\hline
&&\\[-4mm]
\ee_{2} & \ee_{3} - \ee_{2} & qZ_{2} + 1 \\[1mm]
\hline
&&\\[-4mm]
\ee_{3} & \ee_{1} - \ee_{3} &  qZ_{3} + 1 \\[1mm]
\hline 
&&\\[-4mm]
\ee_{4} & \ee_{3} - \ee_{4} &  qZ_{4} + 1 \\[1mm]
\hline
&&\\[-4mm]
\ee_{1} + \ee_{2} & -\ee_{1} &  qZ^{\ee_{1} + \ee_{2}} + qZ_{1} + 1 \\[1mm]
\hline
&&\\[-4mm]
\ee_{1} + \ee_{3} & -\ee_{3} &  qZ^{\ee_{1} + \ee_{3}} + qZ_{3} + 1 \\[1mm]
\hline 
&&\\[-4mm]
\ee_{2} + \ee_{3} & -\ee_{2} &  qZ^{\ee_{2} + \ee_{3}} + qZ_{2} + 1 \\[1mm] 
\hline
&&\\[-4mm]
\ee_{3} + \ee_{4} & \ee_{1} - \ee_{4} &  qZ^{\ee_{3} + \ee_{4}} + qZ_{4} + 1 \\[1mm]
\hline
&&\\[-4mm]
\ee_{1} + \ee_{3} + \ee_{4} &  - \ee_{4} &  qZ^{\ee_{1} + \ee_{3} + \ee_{4}} + qZ^{\ee_{3} + \ee_{4}}  + qZ_{4} + 1 \\[1mm]
\hline
&&\\[-4mm]
\ee_{2} + \ee_{3} + \ee_{4} & -\ee_{2} + \ee_{3} - \ee_{4} & qZ^{\ee_{2} + \ee_{3} + \ee_{4}} + q^{2}Z^{\ee_{2} + \ee_{4}} + qZ_{2} + qZ_{4} + 1 \\[1mm]
\hline
\end{array}
\end{equation*}
\caption{}
\label{table:F-polys}
\end{table}

\section*{Acknowledgments}
The author would like to thank her advisor Andrei Zelevinsky for his many helpful comments and suggestions.

\end{document}